\documentclass[a4paper,draft,11pt]{amsart}

\usepackage{version}
	\excludeversion{original}\includeversion{arxiv}

\usepackage[utf8]{inputenc}
\usepackage[T1]{fontenc}

\usepackage{txfonts}
\usepackage{mathrsfs}

\usepackage{amsmath,amssymb,amsthm}
	\theoremstyle{plain}
	\newtheorem{thm}{Theorem}[section]
	\newtheorem{lem}[thm]{Lemma}
	
	\newtheorem{prop}[thm]{Proposition}
	
	\theoremstyle{definition}
	\newtheorem{dfn}[thm]{Definition}

	\theoremstyle{remark}
	\newtheorem{rem}[thm]{Remark}
	\newtheorem{expl}[thm]{Example}

\usepackage{enumitem}
	\setlist[enumerate,1]{leftmargin=0pt,itemindent=17pt,label=\textup{(\arabic*)}}
	\setlist[itemize]{leftmargin=*}

\usepackage{mathtools}

\usepackage[all]{xy}

\begin{document}
\title{A Fraïssé theoretic approach to the Jiang--Su algebra}
\author[S.~Masumoto]{Shuhei MASUMOTO}
\address[S.~Masumoto]{Graduate~School~of~Mathematical=Sciences, the~University~of~Tokyo}
\email{masumoto@ms.u-tokyo.ac.jp}

\begin{abstract}
In this paper, we give a Fraïssé theoretic proof of the result of X.~Jiang and H.~Su that 
the Jiang--Su algebra is the unique monotracial simple C*-algebra among 
all inductive limits of prime dimension drop algebras.  
The proof presented here is self-contained and quite elementary, 
and does not depend on any K-theoretic technology.  
We also partially recover the fact that every unital endomorphism of 
the Jiang--Su algebra is approximately inner.  
\end{abstract}

\maketitle


\section{Introduction}

The Jiang--Su algebra was originally introduced by Jiang and Su 
in~\cite{jiang99:_simple_unital} 
as a C*-algebraic analog of the hyperfinite type $\mathrm{II}_1$-factor.  
This algebra is characterized as the unique simple monotracial C*-algebra among 
all inductive limits of prime dimension drop algebras 
(i.e.~certain algebras of matrix-valued continuous functions on 
the closed interval $[0, 1]$).  
In addition to being simple and having unique tracial state, 
it is separable, nuclear and infinite-dimensional, 
and have the same K-theory as the complex numbers $\mathbb{C}$.  
Furthermore, every unital endomorphism of this algebra is approximately inner, 
and it tensorially absorbs itself, i.e.~$\mathcal{Z} \otimes \mathcal{Z} \simeq \mathcal{Z}$.  
Because of these properties, it plays a central role in 
the Elliott's classification program of separable nuclear C*-algebras 
via K-theoretic invariants~\cite{elliott08:_regularity_properties}.  

The degree of difficulty in proving the properties of the Jiang--Su algebra 
varies from one to another.  It is immediate from the construction that 
the Jiang--Su algebra is a unital simple separable nuclear infinite-dimensional C*-algebra 
with a unique tracial state and has the same K-theory as $\mathbb{C}$.  
Compared with that, the proofs of the other properties 
that were presented in~\cite{jiang99:_simple_unital} are relatively difficult.  
For example, the uniqueness result of the Jiang--Su algebra among the inductive limits of 
prime dimension drop algebras is implied as a corollary of the complete classification of 
the unital infinite-dimensional simple inductive limits of 
direct sums of dimension drop algebras 
via K-groups and trace spaces~\cite[Theorem~6.2]{jiang99:_simple_unital}.  
The construction of isomorphisms in the classification result is carried over 
by the standard approximate intertwining argument, and for this, 
one has to find an appropriate sequence of morphisms between direct sums of dimension drop algebras 
which is convergent to the isomorphism between the limits.  
This sequence of morphisms are obtained by studying the KK-theory of dimension drop algebras 
and observing how to lift a morphism from a KK-element.  
Such an observation is also used to prove that every unital endomorphism of the Jiang--Su algebra 
is approximately inner.  

It is true that, since KK-theory is a powerful tool, calculating KK-groups of 
dimension drop algebras is significant, 
and that the classification theorem is fairly interesting on its own right.  
On the other hand, it would be natural to ask whether there are more elementary proofs of 
these properties of the Jiang--Su algebra, 
taking the fundamental importance of the algebra into account.  

In this paper, we give an alternative proof that 
a simple monotracial inductive limit of prime dimension drop algebras is 
unique up to isomorphism and every unital endomorphism of 
this inductive limit is approximately inner.  
The proof does not rely on KK-theory, 
but instead it uses Fraïssé theory.  

Fraïssé theory is a topic in model theory 
where a bijective correspondence between certain classes consisting of 
finitely generated structures and countable structures 
with a certain homogeneity property is established.  
By definition, a countable structure is said to be \emph{ultra-homogeneous} if 
every isomorphism between its two finitely generated substructures 
extends to an automorphism.  
The class corresponding to an ultra-homogeneous structure $\mathcal{M}$ 
in the context of Fraïssé theory is 
$\operatorname{Age} M$, the class of all finitely generated structures 
which are embeddable into $\mathcal{M}$.  
Fraïssé's theorem characterizes the classes obtained in this way, and produces 
a method to recover the original ultra-homogeneous structures.  
Namely, if $\{\iota_{n, m} \colon \mathcal{A}_m \to \mathcal{A}_n\}$ and 
$\{\eta_{n, m} \colon \mathcal{B}_m \to \mathcal{B}_n\}$ are 
\emph{generic} inductive systems of 
members of such a class, then, passing to subsystems if necessary, 
one can find sequences $\{\varphi_n \colon \mathcal{A}_n \to \mathcal{B}_n\}$ 
and $\{\psi_n \colon \mathcal{B}_n \to \mathcal{A}_{n+1}\}$ of embeddings 
such that the triangles in the following diagram commute.  
\[
\begin{original}
\begin{tikzcd}
	\mathcal{A}_1 \arrow[r, "\iota_{2,1}"] \arrow[d, "\varphi_1"] & 
	\mathcal{A}_2 \arrow[r, "\iota_{3,2}"] \arrow[d, "\varphi_2"] & 
	\mathcal{A}_3 \arrow[r, "\iota_{4,3}"] \arrow[d, "\varphi_3"] & 
	\mathcal{A}_4 \arrow[r, "\iota_{5,4}"] \arrow[d, "\varphi_4"] & \dots \\
	\mathcal{B}_1 \arrow[r, "\eta_{2,1}"'] \arrow[ru, "\psi_1"'] & 
	\mathcal{B}_2 \arrow[r, "\eta_{3,2}"'] \arrow[ru, "\psi_2"'] & 
	\mathcal{B}_3 \arrow[r, "\eta_{4,3}"'] \arrow[ru, "\psi_3"'] & 
	\mathcal{B}_4 \arrow[r, "\eta_{5,4}"'] \arrow[ru, "\psi_4"'] & \dots 
\end{tikzcd}
\end{original}
\begin{arxiv}
\xymatrix{
	\mathcal{A}_1 \ar[r]^{\iota_{2,1}} \ar[d]^{\varphi_1} & 
	\mathcal{A}_2 \ar[r]^{\iota_{3,2}} \ar[d]^{\varphi_2} & 
	\mathcal{A}_3 \ar[r]^{\iota_{4,3}} \ar[d]^{\varphi_3} & 
	\mathcal{A}_4 \ar[r]^{\iota_{5,4}} \ar[d]^{\varphi_4} & \dots \\
	\mathcal{B}_1 \ar[r]_{\eta_{2,1}} \ar[ru]_{\psi_1} & 
	\mathcal{B}_2 \ar[r]_{\eta_{3,2}} \ar[ru]_{\psi_2} & 
	\mathcal{B}_3 \ar[r]_{\eta_{4,3}} \ar[ru]_{\psi_3} & 
	\mathcal{B}_4 \ar[r]_{\eta_{5,4}} \ar[ru]_{\psi_4} & \dots
}
\end{arxiv}
\]
Consequently, the sequences $\{\varphi_n\}$ and $\{\psi_n\}$ provide 
embeddings between the inductive limits 
which are inverses of each other and so are isomorphisms.  
In fact, the resulting inductive limits are isomorphic to 
the original ultra-homogeneous structure.  
The class and the ultra-homogeneous structure in this context are called 
a Fraïssé class and its Fraïssé limit respectively.  

Like other topics in model theory, this theory has been a target of generalization to 
the setting of metric structures (\cite{schoretsanitis07:_fraisse_theory}, 
\cite{yaacov15:_fraisse_limits}).  
In~\cite{eagle16:_fraisse_limits}, a variant of the theory 
presented in~\cite{yaacov15:_fraisse_limits} was used 
to recognize a part of AF algebras including 
the UHF algebras, the hyperfinite type $\mathrm{II}_1$ factor, and 
the Jiang--Su algebra as generic limits of suitable Fraïssé classes.  
The author also gave an alternative proof of the same result on the Jiang--Su algebra 
as \cite{eagle16:_fraisse_limits}, and realized the UHF algebras 
as Fraïssé limit in a different way~\cite{masumoto16:_jiang_su}.  

The idea in~\cite{eagle16:_fraisse_limits} and~\cite{masumoto16:_jiang_su} 
of realizing the Jiang--Su algebra as a Fraïssé limit is essentially coming from 
the observation that the approximate intertwining argument, 
which is used to construct an isomorphism in 
the classification result for inductive limits of direct sums of 
dimension drop algebras, is a variant of the back-and-forth argument.  
That is, in order to construct suitable sequences of embeddings between 
direct sums of dimension drop algebras, it would be suffice to 
verify that a suitable class of direct sums of dimension drop algebras forms 
a Fraïssé class, which does not necessarily mean that one needs KK-theory.  

Thus, in short, the result on the Jiang--Su algebra in~\cite{eagle16:_fraisse_limits} 
and~\cite{masumoto16:_jiang_su} seems to be a Fraïssé theoretic expression of 
the fact that the algebra is the unique simple monotracial inductive limits of 
prime dimension drop algebras.  Therefore, this fact should be recovered from 
the Fraïssé limit construction, and that is our strategy in this paper.  

This paper consists of three sections.  
In the next section we briefly give a description of Fraïssé theory for 
metric structures.  The Fraïssé limit construction of the Jiang--Su algebra is 
analyzed in the last section.


\section{Approximate isomorphisms and Fraïssé limits}

In this section, we briefly sketch the theory of Fraïssé limits 
used in~\cite{eagle16:_fraisse_limits} and~\cite{masumoto16:_jiang_su}.  
It is a slight modified version of the theory 
presented in~\cite{yaacov15:_fraisse_limits}, 
and detailed proofs can be found in~\cite{masumoto16:_generalized_fraisse}.  

By definition, a \emph{language} is a set $L$ 
such that each element is exactly one of the following symbols:  
\begin{itemize}
\item
	constant symbols; 
\item
	$n$-ary function symbols ($n = 1, 2, 3, \dots$); 
\item
	$n$-ary predicate symbols ($n = 1, 2, 3, \dots$).  
\end{itemize}
A \emph{metric $L$-structure} $\mathcal{A}$ consists of 
a complete metric space $A$, which is called the \emph{domain} of $\mathcal{A}$ 
and denoted by $|\mathcal{A}|$,  
together with an interpretation $S \mapsto S^\mathcal{A}$ of the symbols in $L$.  
\begin{itemize}
\item
	If $c$ is a constant symbol, then $c^\mathcal{A}$ is an element in $A$.  
\item
	If $f$ is an $n$-ary function symbol, then $f^\mathcal{A}$ is 
	a continuous function from $|\mathcal{A}|^n$ into $|\mathcal{A}|$.  
\item
	If $R$ is an $n$-ary predicate symbol, then $R^\mathcal{A}$ is 
	a continuous $\mathbb{R}$-valued function on $|\mathcal{A}|^n$.  
\end{itemize}
An isometry $\iota$ from a metric $L$-structure $\mathcal{A}$ into 
another $L$-structure $\mathcal{B}$ is called 
an \emph{($L$-)embedding} if 
\[
	f^\mathcal{B}\bigl(\iota(a_1), \dots, \iota(a_n)\bigr) 
	= \iota\bigl(f^\mathcal{A}(a_1, \dots, a_n)\bigr) 
	\quad (a_1, \dots, a_n \in |\mathcal{A}|)
\]
for any $n$-ary function symbol $f$ in $L$, and 
\[
	P^\mathcal{B}\bigl(\iota(a_1), \dots, \iota(a_n)\bigr) 
	= P^\mathcal{A}(a_1, \dots, a_n) 
	\quad (a_1, \dots, a_n \in |\mathcal{A}|)
\]
for any $n$-ary predicate symbol $P$ in $L$.  

In this paper, we are interested in the language $L_{\mathrm{TC}^*}$ of 
unital tracial C*-algebras, which consists of the following:  
\begin{itemize}
\item
	two constant symbols $0$ and $1$; 
\item
	an unary function symbol $\lambda$ for each $\lambda \in \mathbb{C}$, 
	which are to be interpreted as multiplication by $\lambda$; 
\item
	an unary function symbol $*$ for involution; 
\item
	a binary function symbol $+$ and ${}\cdot{}$; 
\item
	an unary predicate symbol $\operatorname{tr}$.  
\end{itemize}
Then every unital C*-algebra with a distinguished trace can be considered as 
a metric $L_{\mathrm{TC}^*}$-structure.  
Note that the distance we adopt is the norm distance, 
and that a map between unital C*-algebras with fixed traces are 
$L$-embeddings if and only if 
it is a trace-preserving injective $*$-homomorphism.  

\begin{dfn}\label{dfn:_jep_and_nap}
Let $\mathscr{K}$ be a category of finitely generated separable 
metric $L$-structures and $L$-embeddings.  
\begin{enumerate}
\item
	For $\mathcal{A}, \mathcal{B} \in \mathscr{K}$, we set 
	\[
		\operatorname{JE}_\mathscr{K}(\mathcal{A}, \mathcal{B}) 
		:= \bigl\{(\iota,\eta) \bigm| 
		\exists \mathcal{C} \in \operatorname{Obj}(\mathscr{K}), \ 
		\iota \in \operatorname{Mor}_\mathscr{K}(\mathcal{A}, \mathcal{C}) \ \& \ 
		\eta \in \operatorname{Mor}_\mathscr{K}(\mathcal{B}, \mathcal{C}) \bigr\} 
	\]
	and call each member of $\operatorname{JE}_\mathscr{K}(\mathcal{A},\mathcal{B})$ 
	a \emph{joint $\mathscr{K}$-embedding} of $\mathcal{A}$ and $\mathcal{B}$.  
	The category $\mathscr{K}$ is said to satisfy the \emph{joint embedding property (JEP)} 
	if $\operatorname{JE}_\mathscr{K}(\mathcal{A},\mathcal{B})$ is nonempty 
	for any objects $\mathcal{A}, \mathcal{B}$ of $\mathscr{K}$.  
\item
	The category $\mathscr{K}$ is said to satisfy the \emph{near amalgamation property} 
	if for any objects $\mathcal{A}, \mathcal{B}_1, \mathcal{B}_2$ of $\mathscr{K}$, 
	any morphisms $\iota_i \colon \mathcal{A} \to \mathcal{B}_i$, 
	any finite subset $G \subseteq \mathcal{A}$ and any $\varepsilon > 0$, 
	there exists a joint $\mathscr{K}$-embedding
	$(\eta_1, \eta_2) \in \operatorname{JE}_\mathscr{K}(\mathcal{B}_1,\mathcal{B}_2)$ 
	such that the inequality 
	\[
		d\bigl(\eta_1 \circ \iota_1(a), \eta_2 \circ \iota_2(a) \bigr) \leq \varepsilon
	\]
	holds for all $a \in G$.  
\end{enumerate}
\end{dfn}

In the sequel, we fix a category $\mathscr{K}$ of finitely generated separable 
metric $L$-structures and $L$-embeddings with JEP and NAP.  

\begin{dfn}
\begin{enumerate}
\item
	Let $\mathcal{A}, \mathcal{B}$ be objects of $\mathscr{K}$ and 
	$\varphi \colon |\mathcal{A}| \times |\mathcal{B}| \to [0, \infty]$ be a 
	bi-Kat\v{e}tov map, that is, a map satisfying
	\[
	\begin{aligned}
		\varphi(a, b) &\leq d(a, a') + \varphi(a', b), & 
		d(a, a') &\leq \varphi(a, b) + \varphi(a', b), \\
		\varphi(a, b) &\leq d(b, b') + \varphi(a, b'), &
		d(b, b') &\leq \varphi(a, b) + (a, b')
	\end{aligned}
	\]
	for all $a, a' \in |\mathcal{A}|$ and $b, b' \in |\mathcal{B}|$.  
	Then $\varphi$ is called an \emph{approximate $\mathscr{K}$-isomorphism} 
	from $\mathcal{A}$ to $\mathcal{B}$ 
	if for any finite subsets $A_0 \subseteq |\mathcal{A}|$ and 
	$B_0 \subseteq |\mathcal{B}|$ and any $\varepsilon > 0$, 
	there exists $(\iota, \eta) \in \operatorname{JE}_\mathscr{K}(\mathcal{A}, \mathcal{B})$ 
	such that the inequality 
	\[
		d\bigl(\iota(a), \eta(b)\bigr) \leq \varphi(a, b) + \varepsilon
	\]
	holds for all $a \in A_0$ and $b \in B_0$.  
	We denote by $\operatorname{Apx}_\mathscr{K}(\mathcal{A},\mathcal{B})$ 
	the set of all approximate $\mathscr{K}$-isomorphisms 
	from $\mathcal{A}$ to $\mathcal{B}$.  
\item
	A \emph{$\mathscr{K}$-structure} is a metric $L$-structure $\mathcal{M}$ 
	together with a distinguished inductive system 
	\[
	\begin{original}
	\begin{tikzcd}
		\mathcal{A}_1 \arrow[r,"\iota_1"] & \mathcal{A}_2 \arrow[r,"\iota_2"] &
		\mathcal{A}_3 \arrow[r,"\iota_3"] & \cdots
	\end{tikzcd}
	\end{original}
	\begin{arxiv}
	\xymatrix{
		\mathcal{A}_1 \ar[r]^{\iota_1} & \mathcal{A}_2 \ar[r]^{\iota_2} &
		\mathcal{A}_3 \ar[r]^{\iota_3} & \cdots
	}
	\end{arxiv}
	\]
	in $\mathscr{K}$ such that the inductive limit of the system as 
	a metric $L$-structure is $\mathcal{M}$.  
\item
	By definition, an approximate $\mathscr{K}$-isomorphism from a $\mathscr{K}$-structure 
	$\mathcal{M} = \overline{\bigcup_n \mathcal{A}_n}$ to 
	another $\mathscr{K}$-structure $\mathcal{N} = \overline{\bigcup_m \mathcal{B}_m}$ is 
	a bi-Kat\v{e}tov map $\varphi \colon |\mathcal{M}| \times |\mathcal{N}| \to [0,\infty]$ 
	such that the restriction 
	$\varphi|_{\mathcal{A}_n \times \mathcal{B}_m}$ is in 
	$\operatorname{Apx}_\mathscr{K}(\mathcal{A}_n, \mathcal{B}_m)$ for all $n$ and $m$.  
	We denote by $\operatorname{Apx}_\mathscr{K}(\mathcal{M}, \mathcal{N})$ 
	the set of all approximate $\mathscr{K}$-isomorphisms 
	from $\mathcal{M}$ to $\mathcal{N}$.  
\end{enumerate}
\end{dfn}

The following two examples of approximate $\mathscr{K}$-isomorphisms are 
of most importance.  

\begin{expl}
\begin{enumerate}
\item
	For $L$-embeddings $\iota \colon \mathcal{A} \to \mathcal{C}$ and 
	$\eta \colon \mathcal{B} \to \mathcal{C}$, we set 
	\[
	\begin{aligned}
		\varphi_{\iota,\eta}(a, b) &:= 
		d\bigl(\iota(a), \eta(b)\bigr) & 
		(a \in |\mathcal{A}|, \ b \in |\mathcal{B}|).  
	\end{aligned}
	\]
	If both $\iota$ and $\eta$ are morphisms of $\mathscr{K}$, 
	then $\varphi_{\iota,\eta}$ is an approximate $\mathscr{K}$-isomorphism 
	between objects of $\mathscr{K}$.  
	It is simply written as $\varphi_\iota$ when 
	$\mathcal{C}$ is equal to $\mathcal{B}$ and $\eta$ is the identity map.  
	
	Even if $\iota$ and $\eta$ are not morphisms of $\mathscr{K}$, 
	the bi-Kat\v{e}tov map $\varphi_{\iota,\eta}$ can be 
	an approximate $\mathscr{K}$-isomorphism.  
	An $L$-embedding $\iota$ between $\mathscr{K}$-structures is 
	said to be \emph{$\mathscr{K}$-admissible} if 
	$\varphi_\iota$ is an approximate $\mathscr{K}$-isomorphism.  
\item
	Let $\mathcal{M}_1 = \overline{\bigcup_l \mathcal{A}_l}$, 
	$\mathcal{M}_2 = \overline{\bigcup_m \mathcal{B}_m}$ and 
	$\mathcal{M}_3 = \overline{\bigcup_n \mathcal{C}_n}$ be 
	$\mathscr{K}$-structures.  
	If $\varphi$ is an approximate $\mathscr{K}$-isomorphism from 
	$\mathcal{M}_1$ to $\mathcal{M}_2$ and 
	$\psi$ is an approximate $\mathscr{K}$-isomorphism from 
	$\mathcal{M}_2$ to $\mathcal{M}_3$, 
	then 
	\[
	\begin{aligned}
	\psi\varphi(a, c) &:= \inf_{b \in |\mathcal{M}_2|} 
	\bigl[\varphi(a, b) + \psi(b, c)\bigr] & 
	(a \in |\mathcal{M}_1|, \ c \in |\mathcal{M}_2|) 
	\end{aligned}
	\]
	is an approximate $\mathscr{K}$-isomorphism from 
	$\mathcal{M}_1$ to $\mathcal{M}_3$.  
	For a proof, see~\cite[Proposition~3.4]{masumoto16:_generalized_fraisse}.  
\end{enumerate}
\end{expl}

Now, for each $n \in \mathbb{N}$, set
\[
	\mathscr{K}_n 
	:= \bigl\{\langle \mathcal{A}, \bar{a} \rangle \bigm| 
	\mathcal{A} \in \operatorname{Obj}(\mathscr{K}) \ \& \ 
	\text{$\bar{a} \in |\mathcal{A}|^n$ is an ordered generator of $\mathcal{A}$} \bigr\} 
\]
and define a function $d^\mathscr{K}$ on $\mathscr{K}_n$ by 
\[
\begin{aligned}
	d\bigl(\langle \mathcal{A},\bar{a} \rangle, \langle \mathcal{B},\bar{b} \rangle\bigr) 
	:=& \inf \{ \max_i d\bigl(\iota(a_i),\eta(b_i)\bigr) \mid 
	{(\iota,\eta) \in \operatorname{JE}_\mathscr{K}(\mathcal{A},\mathcal{B})} \} \\
	=& \inf \{ \max_i \varphi(a_i,b_i) \mid 
	{\varphi \in \operatorname{Apx}_\mathscr{K}(\mathcal{A},\mathcal{B})} \}, 
\end{aligned}
\]
where $\bar{a} = (a_1, \dots, a_n)$ and $\bar{b} = (b_1, \dots b_n)$.  
Then it follows from JEP and NAP that $d^\mathscr{K}$ is a pseudo-metric on $\mathscr{K}_n$.  

\begin{dfn}
\begin{enumerate}
\item
	The category $\mathscr{K}$ is said to satisfy the \emph{weak Polish property (WPP)} 
	if $\mathscr{K}_n$ is separable with respect to the pseudo-metric $d^\mathscr{K}$ 
	for all $n \in \mathbb{N}$.  
\item
	The class $\mathscr{K}$ satisfies the \emph{Cauchy continuity property (CCP)} if 
	\begin{enumerate}[label=\alph*)]
	\item
		for any $n$-ary predicate symbol $P$ in $L$, 
		the map $\bigl\langle\mathcal{A},(\bar{a},\bar{b})\bigr\rangle 
		\mapsto P^\mathcal{A}(\bar{a})$ 
		from $\mathscr{K}_{n+m}$ into $\mathbb{R}$ sends 
		Cauchy sequences to Cauchy sequences; and
	\item
		for any $n$-ary function symbol $f$ in $L$, 
		the map $\bigl\langle\mathcal{A},(\bar{a},\bar{b})\bigr\rangle \mapsto 
		\bigl\langle\mathcal{A},(\bar{a},\bar{b},f^\mathcal{A}(\bar{a})\bigr\rangle$ 
		from $\mathscr{K}_{n+m}$ into $\mathscr{K}_{n+m+1}$ sends 
		Cauchy sequences to Cauchy sequences.  
	\end{enumerate}
\end{enumerate}
\end{dfn}

\begin{dfn}
A category $\mathscr{K}$ of finitely generated separable 
metric $L$-structures and $L$-embeddings is called 
a \emph{Fraïssé category} if it satisfies JEP, NAP, WPP and CCP.  
A \emph{limit} of a Fraïssé category $\mathscr{K}$ is 
a $\mathscr{K}$-structure $\mathcal{M}$ with the following properties. 
\begin{enumerate}
\item
	\emph{$\mathcal{M}$ is $\mathscr{K}$-universal}: 
	Every member of $\mathscr{K}$ is $\mathscr{K}$-admissibly embeddable into $\mathcal{M}$.  
\item
	\emph{$\mathcal{M}$ is approximately $\mathscr{K}$-homogeneous}: 
	For any member $\mathcal{A}$ of $\mathscr{K}$, 
	any $\mathscr{K}$-admissible $L$-embeddings 
	$\iota, \eta \colon \mathcal{A} \to \mathcal{M}$, any finite subset 
	$A_0 \subseteq |\mathcal{A}|$ and any $\varepsilon > 0$, 
	there exists a $\mathscr{K}$-admissible $L$-automorphism 
	$\alpha$ of $\mathcal{M}$ such that 
	$d\bigl(\alpha \circ \iota(a), \eta(a)\bigr) < \varepsilon$ holds for all $a \in A_0$.  
\end{enumerate}
\end{dfn}

\begin{thm}
Every Fraïssé category admits a limit.  
Moreover, its limit is unique up to $\mathscr{K}$-admissible isomorphisms.  
\end{thm}

We conclude this section by characterizing the limit of a Fraïssé category 
in terms of approximate isomorphisms.  
For this, we need to introduce the concepts of strictness and totality.  

Let $A \subseteq A'$ and $B \subseteq B'$ be 
subspaces of metric spaces 
and $\varphi \colon A \times B \to [0, \infty]$ be a bi-Kat\v{e}tov map.  
Then the \emph{trivial extension} $\varphi|^{A' \times B'}$ of 
$\varphi$ is 
the bi-Kat\v{e}tov map on $A' \times B'$ defined by
\[
\begin{aligned}
	\varphi|^{A' \times B'}(a',b') &:= 
	\inf_{\substack{a \in A \\ b \in B}} 
	\bigl[d(a',a)+\varphi(a,b)+d(b,b')\bigr] & 
	(a' \in A', \ b' \in B').  
\end{aligned}
\]
Note that if $\psi$ is another bi-Kat\v{e}tov map on $A' \times B'$, 
then $\psi \leq \varphi|^{A' \times B'}$ is equivalent to 
$\psi|_{A \times B} \leq \varphi$.  

\begin{dfn}
Let $\varphi$ be an approximate $\mathscr{K}$-isomorphism 
from $\mathcal{A}$ to $\mathcal{B}$.  
\begin{enumerate}
\item
	The approximate $\mathscr{K}$-isomorphism $\varphi$ is
	said to be \emph{strict} if there exist another approximate isomorphism 
	$\psi$ from $\mathcal{A}$ to $\mathcal{B}$, a positive real number $\varepsilon > 0$ 
	and finite subsets $A_0 \subseteq |\mathcal{A}|$ and $B_0 \subseteq |\mathcal{B}|$ 
	such that the inequality 
	\[
		(\psi|_{A_0 \times B_0})|^{\mathcal{A} \times \mathcal{B}}(a, b) + \varepsilon 
		\leq \varphi(a, b)
	\]
	holds for all $a \in |\mathcal{A}|$ and $b \in |\mathcal{B}|$.  
\item
	The approximate $\mathscr{K}$-isomorphism $\varphi$ is  
	\emph{$\varepsilon$-total} on a subset $A_0 \subseteq |\mathcal{A}|$ if the inequality 
	\[
		\varphi^*\varphi(a, a') \leq d(a, a') + 2\varepsilon, 
	\]
	holds for all $a, a' \in A_0$, or equivalently, 
	\[
		\inf_{b \in \mathcal{B}} \varphi(a, b) \leq \varepsilon
	\]
	for all $a \in A_0$.  
\end{enumerate}
\end{dfn}

The following theorem is a weaker version 
of~\cite[Lemma~4.6]{masumoto16:_generalized_fraisse}.  

\begin{thm}\label{thm:_recognition_of_the_limit}
Let $\mathcal{M} = \overline{\bigcup_n \mathcal{A}_n}$ be a $\mathscr{K}$-structure.  
Then $\mathcal{M}$ is the Fraïssé limit of $\mathscr{K}$ if and only if 
for any object $\mathcal{B}$ of $\mathscr{K}$, 
any strict approximate $\mathscr{K}$-isomorphism $\varphi$ 
from $\mathcal{B}$ to $\mathcal{A}_n$, any finite subset $B_0 \subseteq |\mathcal{B}|$ 
and any $\varepsilon > 0$, there exist $N > n$ and 
an approximate $\mathscr{K}$-isomorphism $\psi$ 
from $\mathcal{B}$ into $\mathcal{A}_N$ 
which is $\varepsilon$-total on $B_0$ and 
dominated by $\varphi|^{\mathcal{B} \times \mathcal{A}_N}$.  
\end{thm}


\section{The Jiang--Su algebra}

In this paper, we shall denote by $\mathbb{M}_n$ the C*-algebra of 
all $n$-by-$n$ complex matrices.  
For natural numbers $p$ and $q$, 
the \emph{dimension drop algebra} $\mathcal{Z}_{p,q}$ is defined as 
the C*-algebra of all $\mathbb{M}_p \otimes \mathbb{M}_q$-valued 
continuous functions $f$ on the closed interval $[0,1]$ such that 
$f(0)$ and $f(1)$ are contained in 
$\mathbb{M}_p \otimes 1_q$ and $1_p \otimes \mathbb{M}_q$ respectively.  
It is said to be \emph{prime} if $p$ and $q$ are coprime.  
Note that if $(e_{ij})_{i,j}$ and $(f_{kl})_{k,l}$ are systems of matrix units of 
$\mathbb{M}_p$ and $\mathbb{M}_q$ respectively, 
then $(e_{ij} \otimes f_{kl})_{(i,k),(j,l)}$ is a system of matrix units, 
so $\mathbb{M}_p \otimes \mathbb{M}_q$ is isomorphic to $\mathbb{M}_{pq}$.  

We denote by $c_p^{p,q}$ the map from $\mathbb{M}_p \otimes 1_q$ to $\mathbb{M}_p$ 
defined by $a \otimes 1 \mapsto a$.  
The map $c_q^{p,q} \colon 1_p \otimes \mathbb{M}_q \to \mathbb{M}_q$ is defined similarly.   
When no confusion arises, these maps are simply denoted by $c$.  
Also, for $t \in [0,1]$, we denote by $e_t$ the evaluation map at $t$.  

The following proposition is a trivial modification of 
\cite[Proposition~3.5]{masumoto16:_jiang_su}.  

\begin{prop}\label{prop:_approximate_diagonalizability}
Let $\iota \colon \mathcal{Z}_{p,q} \to \mathcal{Z}_{p',q'}$ be a unital $*$-homomorphism.  
Then the following statements hold.  
\begin{enumerate}
\item \label{prop:_approximate_diagonalizability:itemi}
	There exist integers $a, b$ with $0 \leq a < q$ and $0 \leq b < p$, 
	continuous maps $t_1, \dots, t_k$ from $[0,1]$ into $[0,1]$ and 
	a family $\{v_s\}_{s \in [0,1]}$ of unitary matrices of size $p'q'$ such that 
	$\iota$ is of the form 
	\[
	\begin{aligned}
		\iota\bigl(f\bigr)(s) 
		= \operatorname{Ad}(v_s) \Bigl( \operatorname{diag}
		\bigl[\overbrace{c(f(0)), \dots, c(f(0))}^a, & \\
		f(t_1(s)), \dots, f(t_k(s)), &\underbrace{c(f(1)), \dots, c(f(1))}_b \bigr] \Bigr)
	\end{aligned}
	\]
	for $f \in \mathcal{Z}_{p,q}$ and $s \in [0,1]$, 
	where $\operatorname{Ad}(v)$ denotes the inner automorphism associated to $v$, 
	and $\operatorname{diag}[a_1,\dots,a_n]$ is the block diagonal matrix with $a_i$ as 
	its $i$-th block.  
\item
	Suppose that $t_1, \dots, t_k$ are as in~\ref{prop:_approximate_diagonalizability:itemi}.  
	Then for any finite $G \subseteq \mathcal{Z}_{p,q}$ and any $\varepsilon > 0$, 
	there exists a continuous path $u \colon [0,1] \to \mathbb{M}_{p'q'}$ of unitaries 
	such that the $*$-homomorphism 
	$\iota' \colon \mathcal{Z}_{p,q} \to \mathcal{Z}_{p',q'}$ defined by 
	\[
	\begin{aligned}
		\iota'\bigl(f\bigr)(s) 
		= \operatorname{Ad}(u(s)) \Bigl( \operatorname{diag}
		\bigl[\overbrace{c(f(0)), \dots, c(f(0))}^a, & \\
		f(t_1(s)), \dots, f(t_k(s)), &\underbrace{c(f(1)), \dots, c(f(1))}_b \bigr] \Bigr) 
	\end{aligned}
	\]
	satisfies $\|\iota(g) - \iota'(g)\| < \varepsilon$ for all $g \in G$.  
\end{enumerate}
\end{prop}

The family $t_1, \dots, t_k$ of continuous maps and the integers $a$ and $b$ 
that appeared in Proposition~\ref{prop:_approximate_diagonalizability} are
called an \emph{eigenvalue pattern} and the \emph{remainder indices} of 
the $*$-homomorphism $\iota$.  
An eigenvalue pattern $t_1, \dots, t_k$ is said to be \emph{normalized} 
if it satisfies the inequality $t_1 \leq \dots \leq t_k$.  
Note that the normalized eigenvalue pattern is unique for each $*$-homomorphism.  
Also, if $\mathcal{Z}_{p,q}$ is prime, then the remainder indices depend only on 
the integers $p$, $q$, $p'$ and $q'$.  Indeed, if $\eta$ is another $*$-homomorphism from 
$\mathcal{Z}_{p,q}$ into $\mathcal{Z}_{p',q'}$, and if 
$a_\eta$ and $b_\eta$ are the remainder indices of $\eta$, then
the congruence equation
\[
	pa + qb \equiv p'q' \equiv pa_\eta + qb_\eta \pmod{pq}
\]
holds, so that 
\[
\begin{aligned}
	a &\equiv a_\eta \pmod{q}, & b \equiv b_\eta \pmod{p}, 
\end{aligned}
\]
as $p$ and $q$ are coprime.  

Let $\iota$ be a $*$-homomorphism between dimension drop algebras with 
an eigenvalue pattern $t_1, \dots, t_k$.  
We shall denote by $V(t_1, \dots, t_k)$ the maximum of 
the diameters of the images of $t_1, \dots, t_k$, and 
call it the \emph{variation} of the eigenvalue pattern.  
The infimum of the variations of the eigenvalue patterns is 
clearly equal to the variation of the normalized eigenvalue pattern, 
which is called the variation of $\iota$ and denoted by $V(\iota)$.  
The next proposition is a variant of~\cite[Proposition~4.4]{masumoto16:_jiang_su}.  

\begin{prop}\label{prop:_embeddability_of_dimension_drop_algebras}
Let $p, q \in \mathbb{N}$ be coprime and $\varepsilon$ be a positive real number.  
Then there exists $M \in \mathbb{N}$ such that 
if $p', q'$ are larger than $M$, then there exists a unital embedding of 
$\mathcal{Z}_{p,q}$ into $\mathcal{Z}_{p',q'}$ with its variation less than $\varepsilon$.  
\end{prop}

\begin{proof}
Since $p$ and $q$ are coprime, there exists $M \in \mathbb{N}$ with 
$M \geq pq(1/\varepsilon+2)$ such that if $p', q' > M$, then 
\[
	pa + pqk + qb = p'q'
\]
for some $a, b, k \in \mathbb{N}$.  
Without loss of generality, we may assume $0 \leq a < q$ and $0 \leq b < p$.  
Also, one can find $l^0, m^0 \in [0,q)$ and $l^1, m^1 \in [0,p)$ such that 
\[
\begin{aligned}
	pl^0 & \equiv p' \pmod{q}, & pm^0 & \equiv q' \pmod{q}, \\
	ql^1 & \equiv p' \pmod{p}, & qm^1 & \equiv q' \pmod{p}.  
\end{aligned}
\]
Then, 
\[
\begin{aligned}
	pq'l^0 &\equiv pp'm^0 \equiv p'q' \equiv pa \pmod{q}, \\
	qq'l^1 &\equiv qp'm^1 \equiv p'q' \equiv qb \pmod{p}, 
\end{aligned}
\]
so 
\[
\begin{aligned}
	q'l^0 &\equiv p'm^0 \equiv a \pmod{q}, & q'l^1 \equiv p'm^1 \equiv b \pmod{p}.  
\end{aligned}
\]
We set 
\[
\begin{aligned}
	n_0^0 &:= \frac{q'l^0-a}{q}, & n_0^1 &:= \frac{q'l^1-b}{p}, \\
	n_0^0 &:= \frac{p'm^0-a}{q}, & n_0^1 &:= \frac{p'm^1-b}{p}.  
\end{aligned}
\]
so that 
\[
\begin{aligned}
	a+qn_0^0 &\equiv b + pn_0^1 \equiv 0 \pmod{q'}, \\
	a+qn_1^0 &\equiv b + pn_1^1 \equiv 0 \pmod{p'}.  	
\end{aligned}
\]

We claim that 
\begin{itemize}
\item
	$n_0^0 + n_0^1$ and $n_1^0 + n_1^1$ are smaller than $k$; 
\item
	$k - n_0^0 - n_0^1$ and $k - n_1^0 - n_1^1$ are multiples of $q'$ and $p'$ respectively; 
	and
\item
	$(k - n_0^0 - n_0^1)/q'$ and $(k - n_1^0 - n_1^1)/p'$ are larger than $1/\varepsilon$.  
\end{itemize}
Indeed, we have 
\[
\begin{aligned}
	n_0^0 + n_0^1 
	&= \frac{q'l^0 - a}{q} + \frac{q'l^1 - b}{p} \\
	&= \frac{q'(pl^0+ql^1) - pa - qb}{pq} \\
	&< \frac{2q'pq - p'q' + pqk}{pq} < k.  
\end{aligned}
\]
Also, note that 
\[
	pq(k - n_0^0 - n_0^1) = p'q' - pq'l^0 - qq'l^1 = q'(p' - pl^0 - ql^1).  
\]
Since $p$ and $q$ divide $p'-ql^1$ and $p'-pl^0$ respectively, 
and since $p$ and $q$ are coprime, it follows that $pq$ divides $p' - pl^0 - ql^1$, 
so $q'$ divides $k - n_0^0 - n_0^1$; and 
\[
	\frac{k - n_0^0 - n_0^1}{q'} = \frac{p' - pl^0 - ql^1}{pq} 
	> \frac{p' - 2pq}{pq} > \frac{1}{\varepsilon}.  
\]
Similarly, it follows that $n_1^0 + n_1^1$ is smaller than $k$, 
that $p'$ divides $k - n_1^0 + n_1^1$, 
and that $(k - n_1^0 - n_1^1)/p'$ is larger than $1/\varepsilon$.  

From the claim in the previous paragraph, 
one can easily construct a family $t_1, \dots t_k$ of continuous maps 
from $[0,1]$ into $[0,1]$ such that 
\begin{itemize}
\item
	the union of the images of $t_1, \dots, t_k$ is equal to $[0,1]$; 
\item
	the diameter of the image of $t_i$ is smaller than $\varepsilon$ for all $i$; 
\item
	$\#\{ i \mid t_i(x) = y \} = n_x^y$ for $x,y = 0,1$; and
\item
	for each $y$ with $0 < y < 1$, the integers $q'$ and $p'$ divide 
	$\#\{ i \mid t_i(0) = y \}$ and $\#\{ i \mid t_i(1) = y \}$ respectively.  
\end{itemize}
If we define a $*$-homomorphism $\eta$ from $\mathcal{Z}_{p,q}$ into 
$C([0,1],\mathbb{M}_{p'q'})$ by 
\[
\begin{aligned}
	\eta\bigl(f\bigr)(s) 
	= \Bigl( \operatorname{diag}
	\bigl[\overbrace{c(f(0)), \dots, c(f(0))}^a, & \\
	f(t_1(s)), \dots, f(t_k(s)), &\underbrace{c(f(1)), \dots, c(f(1))}_b \bigr] \Bigr), 
\end{aligned}
\]
then one can easily verify from the construction of $t_1, \dots, t_k$ that 
the images of $e_0 \circ \eta$ and $e_1 \circ \eta$ are included in 
isomorphic copies of $\mathbb{M}_{p'} \otimes 1_{q'}$ and 
$1_{p'} \otimes \mathbb{M}_{q'}$ respectively, 
so there is a unitary $u \in C([0,1],\mathbb{M}_{p'q'})$ 
with $\operatorname{Im} \operatorname{Ad}(u) \circ \eta \subseteq \mathcal{Z}_{p',q'}$.  
\end{proof}

Note that the integers $a$ and $b$ in the proof of the previous proposition is 
the reminder indices of the embedding that is constructed.  
In particular, both of the indices are equal to $0$ if $pq$ divides $p'q'$.  

The next proposition is also a slight modification 
of~\cite[Lemma~4.9]{masumoto16:_jiang_su}.  
Recall that a \emph{modulus of uniform continuity} of a function $f$ on $[0,1]$ is 
a map $\Delta_f \colon (0,1] \to (0,1]$ such that $|s-s'| < \Delta_f(\varepsilon)$ 
implies $\|f(s) - f(s')\| \leq \varepsilon$.  

\begin{prop}\label{prop:_inner_automorphisms}
Let $p, q$ be coprime positive integers, 
$\iota_1, \iota_2 \colon \mathcal{Z}_{p,q} \to \mathcal{Z}_{p',q'}$ be 
unital $*$-homomorphisms with eigenvalue patterns 
$t_1^1, \dots t_k^1$ and $t_1^2, \dots, t_k^2$ respectively, 
$G$ be a finite subset of $\mathcal{Z}_{p,q}$, 
and $\varepsilon$ be a positive real number.  
If the inequality 
\[
	\max_i\|t_i^1 - t_i^2\|_\infty < \min_{g \in G} \Delta_g(\varepsilon)
\]
holds, where $\Delta_g$ is a modulus of uniform continuity of $g$, 
then there exists a unitary $w \in \mathcal{Z}_{p',q'}$ with
\[
	\|\operatorname{Ad}(w) \circ \iota_1(g) - \iota_2(g)\| < 5\varepsilon
\]
for all $g \in G$.  
\end{prop}

\begin{proof}
By Proposition~\ref{prop:_approximate_diagonalizability}, 
we may assume without loss of generality that $\iota_j$ is 
of the form
\[
\begin{aligned}
	\iota_j\bigl(f\bigr)(s) 
	= \operatorname{Ad}(u^j(s)) \Bigl( \operatorname{diag}
	\bigl[\overbrace{c(f(0)), \dots, c(f(0))}^a, & \\
	f(t_1^j(s)), \dots, f(t_k^j(s)), &\underbrace{c(f(1)), \dots, c(f(1))}_b \bigr] \Bigr) 
\end{aligned}	
\]
for $f \in \mathcal{Z}_{p,q}$, where $u^j \in C([0,1],\mathbb{M}_{p'q'})$ 
is a unitary and $t_1^j \leq \dots \leq t_k^j$.  
Also, we may assume that $\|g\| \leq 1$ for all $g \in G$.  

Let $n_0^0$ and $n_0^1$ be the least non-negative integers such that 
\[
	a + qn_0^0 \equiv b + pn_0^1 \equiv 0 \pmod{q'}.  
\]
Then, from the condition $\iota_j(f) \in \mathbb{M}_{p'} \otimes 1_{q'}$, 
it easily follows that 
\[
\begin{aligned}
	0 = t_{n_0^0}^j(0) 
	&\leq t_{n_0^0+1}^j(0) = \dots = t_{n_0^0+q'}^j(0) \\
	&\leq t_{n_0^0+q'+1}^j(0) = \dots = t_{n_0^0+2q'}^j(0) \\
	&\leq \dots \leq t_{k-n_0^1+1}^j(0) = 1, 
\end{aligned}
\]
and there exists a unitary $v_0^j \in \mathbb{M}_{p'}$ such that 
\[
\begin{aligned}
	c\bigl(\iota_j\bigl(f\bigr)(0)\bigr) 
	= \operatorname{Ad}(v_0^j) \Bigl( \operatorname{diag}
	\bigl[&\overbrace{c(f(0)), \dots, c(f(0))}^{a'}, \\
	&f\bigl(t_{n_0^0 + q'}^j(0)\bigr), f\bigl(t_{n_0^0 + 2q'}^j(0)\bigr), \dots, 
	f\bigl(t_{k-n_0^1}^j(0)\bigr), \\
	&\qquad\qquad\qquad\qquad\quad\underbrace{c(f(1)), \dots, c(f(1))}_{b'} \bigr] \Bigr)  
\end{aligned}
\]
for some non-negative integers $a'$ and $b'$.  
Similarly, for suitable non-negative integers $n_1^0, n_1^1, a''$ and $b''$ and 
a unitary $v_1^j \in \mathbb{M}_{q'}$, we have 
\[
\begin{aligned}
	c\bigl(\iota_j\bigl(f\bigr)(1)\bigr) 
	= \operatorname{Ad}(v_1^j) \Bigl( \operatorname{diag}
	\bigl[&\overbrace{c(f(0)), \dots, c(f(0))}^{a''}, \\
	&f\bigl(t_{n_1^0 + p'}^j(1)\bigr), f\bigl(t_{n_1^0 + 2p'}^j(1)\bigr), \dots, 
	f\bigl(t_{k-n_1^1}^j(p)\bigr), \\
	&\qquad\qquad\qquad\qquad\quad\underbrace{c(f(1)), \dots, c(f(1))}_{b''} \bigr] \Bigr).  
\end{aligned}
\]
Thus, if $v$ is a path of unitaries connecting $\bigl[v_0^2(v_0^1)^*\bigr] \otimes 1_{q'}$ 
to $1_{p'} \otimes \bigl[v_1^2(v_1^1)^*\bigr]$, then $v$ is in $\mathcal{Z}_{p',q'}$ and 
$\bigl(\operatorname{Ad}(v) \circ \iota_1\bigr)\bigl(f\bigr)(s) = 
\bigl(\operatorname{Ad}(u^2(u^1)^*) \circ \iota_1\bigr)\bigl(f\bigr)(s)$ 
for $s = 0, 1$.  Therefore, considering $\operatorname{Ad}(v) \circ \iota_1$ 
instead of $\iota_1$ if necessary, 
we may assume from the outset that $u^2(0)u^1(0)^*$ and $u^2(1)u^1(1)^*$ commutes 
with every matrix in the image of $e_0 \circ \iota_1$ and $e_1 \circ \iota_1$.  

Now, take $\delta > 0$ so that $|s-s'| < \delta$ implies 
$|t_i^j(s) - t_i^j(s')| < \min_{g \in G} \Delta_g(\varepsilon)$ and 
$\|u^j(s) - u^j(s')\| < \varepsilon$.  
Let $w \colon [0,1] \to \mathbb{M}_{p'q'}$ be a path of unitaries such that 
\begin{itemize}
\item
	$w|_{[0, \delta/2]}$ connects $1_{p'q'}$ to $u^2(0)u^1(0)^*$ within 
	the commutant of the image of $e_0 \circ \iota_1$; 
\item
	$w(s) = u^2(2s-\delta')u^1(2s-\delta)^*$ for $s \in [\delta/2, \delta]$; 
\item
	$w(s) = u^2(s)u^1(s)^*$ for $s \in [\delta, 1-\delta]$; 
\item
	$w(s) = u^2(2s-1+\delta)u^1(2s-1+\delta)^*$ for $s \in [1-\delta, 1-\delta/2]$; and
\item
	$w|_{[1-\delta/2, 1]}$ connects $u^2(1)u^1(1)^*$ to $1_{p'q'}$ within 
	the commutant of the image of $e_0 \circ \iota_1$.  
\end{itemize}
Then it is not difficult to see that this $w$ has the desired property.  
\end{proof}

Now, let $\mu$ be a probability Radon measure on $[0,1]$.  Then 
\[
\begin{aligned}
	f &\mapsto \int_0^1 \operatorname{tr}(f(t))\, d\mu(t) & (f \in \mathcal{Z}_{p,q})
\end{aligned}
\]
is a tracial state on $\mathcal{Z}_{p,q}$, 
where $\operatorname{tr}$ denotes the normalized trace on $\mathbb{M}_{pq}$.  
One can easily show that every tracial state on a dimension drop algebra is 
of this form.  Henceforth, we identify probability Radon measures on $[0,1]$ with
tracial states on dimension drop algebras and use the same adjectives for 
measures and traces in common.  Thus, for example, a trace is said to be atomless 
if its corresponding measure is atomless.  

The following lemma is useful for dealing with faithful measures on $[0,1]$.  
A proof can be found in~\cite[Lemma~3.1 and Proposition~3.2]{masumoto16:_jiang_su}, 
for example.   

\begin{lem}\label{lem:_transformations_of_measures}
If $\lambda$ is an atomless faithful probability measure on $[0,1]$, 
then for any faithful probability measure $\tau$ on $[0,1]$, 
there exists a non-decreasing continuous surjection $\beta$ from $[0,1]$ onto $[0,1]$ 
with $\beta_*(\lambda) = \tau$, so that $\beta^*$ is a trace-preserving embedding of
$\langle \mathcal{Z}_{p,q}, \tau \rangle$ into $\langle \mathcal{Z}_{p,q}, \lambda \rangle$.  
\end{lem}

We shall define the category $\mathscr{K}_\mathcal{Z}$ as following.  
\begin{itemize}
\item
	$\operatorname{Obj}(\mathscr{K}_\mathcal{Z})$ is the class of 
	all the pairs $\langle \mathcal{Z}_{p,q}, \tau \rangle$, 
	where $p, q$ are coprime and $\tau$ is a faithful tracial state on $\mathcal{Z}_{p,q}$.  
\item
	Every $L_{\mathrm{TC}^*}$-embedding between objects of $\mathscr{K}_\mathcal{Z}$ is 
	a morphism of $\mathscr{K}_\mathcal{Z}$.  
\end{itemize}
The following theorem was first proved in~\cite{eagle16:_fraisse_limits}.  
The proof presented here is essentially the same as 
that of~\cite{masumoto16:_jiang_su}.  

\begin{thm}\label{thm:_dimension_drop_algebras_form_a_fraisse_class}
The category $\mathscr{K}_\mathcal{Z}$ is a Fraïssé category.  
\end{thm}

\begin{proof}
In view of Lemma~\ref{lem:_transformations_of_measures}, 
one can easily modify the proof of 
Proposition~\ref{prop:_embeddability_of_dimension_drop_algebras} 
to show the following claim: For any coprime integers $p$ and $q$ and 
any faithful tracial state $\tau$ on $\mathcal{Z}_{p,q}$, 
there exists a natural number $M$ such that if $p'$ and $q'$ are larger than $M$ 
and $pq$ divides $p'q'$, then we can construct a \emph{trace-preserving} $*$-homomorphism 
from $\langle \mathcal{Z}_{p,q}, \tau \rangle$ 
into $\langle\mathcal{Z}_{p',q'}, \lambda \rangle$, 
where $\lambda$ is any atomless tracial state on $\mathcal{Z}_{p',q'}$.  
(For a precise proof of this claim, 
see~\cite[Proposition~4.4]{masumoto16:_jiang_su}.)  
So $\mathscr{K}_\mathcal{Z}$ satisfies JEP.  
Also, a combination of Propositions 
\ref{prop:_approximate_diagonalizability}, 
\ref{prop:_embeddability_of_dimension_drop_algebras} 
and~\ref{prop:_inner_automorphisms} immediately yields a proof of NAP.  
Next, fix an atomless measure $\lambda$ on $[0,1]$.  
Then any object $\langle \mathcal{Z}_{p,q}, \tau \rangle$ of $\mathscr{K}_\mathcal{Z}$ 
can be embedded into $\langle \mathcal{Z}_{p,q},\lambda \rangle$ 
by Lemma~\ref{lem:_transformations_of_measures}, so WPP follows.  
Finally, CCP is automatic, since all the relevant functions and relations are 
$1$-Lipschitz on the unit ball.  
\end{proof}

Henceforth, we shall denote by $\langle \mathcal{Z}, \operatorname{tr} \rangle$ 
the Fraïssé limit of $\mathscr{K}_\mathcal{Z}$.  
From Theorem~\ref{thm:_characterization_of_jiang_su_algebra} below, 
it follows that the C*-algebra $\mathcal{Z}$ is the same as 
the one constructed in~\cite[Section~2]{jiang99:_simple_unital}, 
the so-called Jiang--Su algebra.  

We shall say an inductive system of prime dimension drop algebras with 
distinguished traces is \emph{regular} 
if its inductive limit is isomorphic to $\langle \mathcal{Z}, \operatorname{tr} \rangle$.  
In the sequel, we shall establish a method of recognizing regular systems.  

\begin{lem}\label{lem:_universal_measures}
Suppose that $p$ and $q$ are coprime and 
$\mathcal{Z}_{p,q}$ is embeddable into $\mathcal{Z}_{p',q'}$.  
Then there exists a tracial state $\lambda_{p',q'}$ on $\mathcal{Z}_{p,q}$ 
with the following properties.  
\begin{enumerate}
\item
	There exists a trace-preserving embedding 
	from $\langle \mathcal{Z}_{p,q}, \lambda_{p',q'} \rangle$ 
	into $\langle \mathcal{Z}_{p',q'}, \lambda \rangle$, where $\lambda$ 
	corresponds to the Lebesgue measure on $[0,1]$.  
\item
	If $\tau$ is a tracial state on $\mathcal{Z}_{p,q}$ of the form $\iota^*(\tau')$ 
	for some embedding $\iota$ of $\mathcal{Z}_{p,q}$ into $\mathcal{Z}_{p',q'}$ 
	and some tracial state $\tau'$ on $\mathcal{Z}_{p'q'}$, 
	then there exists a non-decreasing continuous map $\beta$ from $[0,1]$ onto $[0,1]$ with 
	$\beta_*(\lambda_{p,q}) = \tau$.  
\end{enumerate}
\end{lem}

\begin{proof}
Since $\mathcal{Z}_{p,q}$ is embeddable into $\mathcal{Z}_{p',q'}$, 
there is an embedding $\rho$ of $\mathcal{Z}_{p,q}$ into $\mathcal{Z}_{p',q'}$ of the form 
\[
\begin{aligned}
	\rho\bigl(f\bigr)(s) 
	= \operatorname{Ad}(u(s)) \Bigl( \operatorname{diag}
	\bigl[\overbrace{c(f(0)), \dots, c(f(0))}^a, & \\
	f(t_1(s)), \dots, f(t_k(s)), &\underbrace{c(f(1)), \dots, c(f(1))}_b \bigr] \Bigr), 
\end{aligned}	
\]
where $t_1, \dots, t_k$ are piecewise strictly monotone functions such that 
the union of the images is equal to $[0,1]$.  
We shall set $\lambda_{p',q'} := \rho^*(\lambda)$.  
Note that if $\lambda_{p',q'} = \lambda_{p',q'}^d + \lambda_{p',q'}^c$ is 
the Lebesgue decomposition, then the discrete measure $\lambda_{p',q'}^d$ is equal to 
$(ap\delta_0 + bq\delta_1)/p'q'$, where $\delta_0$ and $\delta_1$ are the Dirac measures 
supported on $\{0\}$ and $\{1\}$ respectively, 
and the support of the atomless measure $\lambda_{p,q}^c$ is $[0,1]$.  

If $\tau$ is of the form $\iota^*(\tau')$ 
for some embedding $\iota$ of $\mathcal{Z}_{p,q}$ into 
$\mathcal{Z}_{p',q'}$ and some tracial state $\tau'$ on $\mathcal{Z}_{p',q'}$, 
then necessarily $\tau = \lambda_{p',q'}^d + \mu$ 
for a suitable measure $\mu$ on $[0,1]$, and $\|\lambda_{p',q'}^c\| = \|\mu\|$.  
Since $\lambda_{p',q'}^c$ is continuous, there exists a non-decreasing continuous map 
$\beta$ from $[0,1]$ onto $[0,1]$ with $\beta_*(\lambda_{p',q'}^c) = \mu$, 
so $\beta_*(\lambda_{p',q'}) = \tau$, as desired.  
\end{proof}

\begin{lem}\label{lem:_perturbation_of_almost_constant_functions}
Let $\iota \colon \mathcal{Z}_{p,q} \to \mathcal{Z}_{p',q'}$ be a 
unital $*$-homomorphism, $\beta \colon [0,1] \to [0,1]$ be 
a non-decreasing continuous surjection, $G$ be a finite subset of $\mathcal{Z}_{p,q}$, 
and $\varepsilon$ be a positive real number.  
Suppose that the inequality $V(\iota) < \min_{g \in G} \Delta_g(\varepsilon)$ holds, 
where $\Delta_g$ denotes a modulus of uniform continuity of $g$.  
Then there exists a unitary $w \in \mathcal{Z}_{p,q}$ with
$\|\bigl(\operatorname{Ad}(w) \circ \beta^* \circ \iota \bigr)(g) - \iota(g)\| 
< 5\varepsilon$ for all $g \in G$.  
\end{lem}

\begin{proof}
Note that if $t_1 \leq \dots \leq t_k$ is the normalized eigenvalue pattern of $\iota$, 
then $\|t_i - t_i \circ \beta\|_\infty < \min_{g \in G} \Delta_g(\varepsilon)$.  
Thus, the claim is immediate from Proposition~\ref{prop:_inner_automorphisms}.  
\end{proof}

At first sight, the proof of the following proposition might seem to be complicated.  
However, the underlying idea is very simple; see Remark~\ref{rem:_idea_of_the_proof}.  

\begin{prop}\label{prop:_variation_and_limit}
An inductive system $\{\langle \mathcal{Z}_{p_n,q_n},\tau_n \rangle,\iota_{n,m}\}$ 
of prime dimension drop algebras with distinguished traces is regular 
if $\lim_{n \to \infty} V(\iota_{n,m}) = 0$ for all $m$.  
\end{prop}

\begin{proof}
We shall apply Theorem~\ref{thm:_recognition_of_the_limit}.  
Let $\langle \mathcal{Z}_{p,q}, \tau \rangle$ be 
a prime dimension drop algebra with a fixed faithful trace, 
$F$ be a finite subset of $\mathcal{Z}_{p,q}$, 
and $\varphi$ be a strict approximate $\mathscr{K}_\mathcal{Z}$-isomorphism 
from $\langle \mathcal{Z}_{p,q}, \tau \rangle$ 
into $\langle \mathcal{Z}_{p_n,q_n}, \tau_n \rangle$.  
Our goal is to find an approximate $\mathscr{K}_\mathcal{Z}$-isomorphism $\psi$ 
from $\langle \mathcal{Z}_{p,q}, \tau \rangle$ 
into $\langle \mathcal{Z}_{p_N,q_N}, \tau_N \rangle$ for some $N > n$ such that 
\begin{itemize}
\item
	$\psi(f,g) \leq \varphi(f,g)$ for $f \in \mathcal{Z}_{p,q}$ and 
	$g \in \mathcal{Z}_{p_n,q_n}$, and 
\item
	$\psi$ is $\varepsilon$-total on $F$ for a given $\varepsilon > 0$.  
\end{itemize}
By the definition of strict approximate isomorphisms, 
there exist finite subsets $G_1 \subseteq \mathcal{Z}_{p,q}$ 
and $G_2 \subseteq \mathcal{Z}_{p_n,q_n}$, 
morphisms $\theta_1, \theta_2$ from $\langle \mathcal{Z}_{p,q}, \tau \rangle$ and 
$\langle \mathcal{Z}_{p_n,q_n}, \tau_n \rangle$ into 
some $\langle \mathcal{Z}_{r,s},\sigma \rangle$, 
and a positive real number $\delta$ such that 
\[
	\varphi \geq (\varphi_{\theta_1,\theta_2}|_{G_1 \times G_2})
	|^{\mathcal{Z}_{p,q} \times \mathcal{Z}_{p_n,q_n}} + \delta.  
\]
Here, we may assume the following.  
Fix an arbitrary positive real number $\gamma$.  
\begin{enumerate}
\item
	The subset $G_1$ includes $F$.  
	This is because we may replace $G_1$ with a larger subset.  
\item
	There exist $m < n$ and a finite subset $G_2' \subseteq \mathcal{Z}_{p_m,q_m}$ 
	such that $\iota_{n,m}[G_2'] = G_2$ and 
	$V(\iota_{n,m}) < \Delta_g(\gamma)$ for all $g \in G_2'$, 
	where $\Delta_g$ is a modulus of uniform continuity 
	for $g$.  This is because, taking our goal into account, we may replace 
	$\varphi$ with $\varphi|^{\mathcal{Z}_{p,q} \times \mathcal{Z}_{p_l,q_l}}$ 
	for $l > n$, and we have
	\[
	\begin{aligned}
		\varphi|^{\mathcal{Z}_{p,q} \times \mathcal{Z}_{p_l,q_l}} 
		&\geq \bigl[(\varphi_{\theta_1,\theta_2}|_{G_1 \times G_2})
		|^{\mathcal{Z}_{p,q} \times \mathcal{Z}_{p_n,q_n}} + \delta\bigr]
		|^{\mathcal{Z}_{p,q} \times \mathcal{Z}_{p_l,q_l}} \\ 
		&= (\varphi_{\theta_1,\theta_2}|_{G_1 \times \iota_{l,n}[G_2]})
		|^{\mathcal{Z}_{p,q} \times \mathcal{Z}_{p_n,q_n}} + \delta.  
	\end{aligned}
	\]
\item
	The embedding $\theta_1$ satisfies $V(\theta_1) < \Delta_f(\gamma)$ for all $f \in G_1$, 
	by Proposition~\ref{prop:_embeddability_of_dimension_drop_algebras}.  
\item
	The tracial state $\sigma$ is atomless, 
	by Lemma~\ref{lem:_transformations_of_measures}.  
\end{enumerate}

Now, take sufficiently large $N$ so that there exists an embedding $\zeta$ of 
$\mathcal{Z}_{r,s}$ into $\mathcal{Z}_{p_N,q_N}$ with 
$V(\zeta) < \Delta_g(\gamma)$ for all $g \in \theta_i[G_i]$.  
Let $\lambda$ be the tracial state on $\mathcal{Z}_{p_N,q_N}$ corresponding to 
the Lebesgue measure, $\alpha$ be the nondecreasing surjective continuous map 
from $[0,1]$ to $[0,1]$ with $\alpha_*(\lambda) = \tau_N$, and 
$\Sigma_\alpha$ be the closed subset of $[0,1]$ such that 
$f \in \mathcal{Z}_{p_N,q_N}$ is in the image of $\alpha^*$ if and only if 
$f$ is constant on $\Sigma_\alpha$.  
Also, let $\lambda_{p_N,q_N}$ be the tracial state on $\mathcal{Z}_{p_n,q_n}$ 
as in Lemma~\ref{lem:_universal_measures}, 
and set 
\[
\begin{aligned}
	\sigma' &:= \zeta^*(\lambda), & 
	\tau' &:= \theta_1^*(\sigma'), & 
	\tau_n' &:= \theta_2^*(\sigma').  
\end{aligned}
\]

By Lemmas \ref{lem:_universal_measures} 
and~\ref{lem:_perturbation_of_almost_constant_functions} and 
assumption~(2) in the first paragraph, 
there exists a morphism $\eta$ from $\langle \mathcal{Z}_{p_n,q_n}, \tau_n \rangle$ 
to $\langle \mathcal{Z}_{p_n,q_n}, \lambda_{p_N,q_N} \rangle$ with 
$\|\eta(g) - g\| < 5\gamma$ for all $g \in G_2$.  
Similarly, there exists a morphism $\eta'$ 
from $\langle \mathcal{Z}_{p_n,q_n}, \tau_n' \rangle$ 
to $\langle \mathcal{Z}_{p_n,q_n}, \lambda_{p_N,q_N} \rangle$ with 
$\|\eta'(g) - g\| < 5\gamma$ for all $g \in G_2$.  
Also, by Lemmas \ref{lem:_transformations_of_measures} 
and~\ref{lem:_perturbation_of_almost_constant_functions} and 
assumption~(3) in the first paragraph, 
there exists a morphism $\rho$ from $\langle \mathcal{Z}_{r,s}, \sigma' \rangle$ 
to $\langle \mathcal{Z}_{r,s}, \sigma \rangle$ with 
$\|\rho(f) - f\| < 5\gamma$ for all $f \in \theta_1[G_1]$.  
Finally, by Lemma~\ref{lem:_universal_measures}, 
one can find a morphism $\iota$ 
from $\langle \mathcal{Z}_{p_n,q_n}, \lambda_{p_N, q_N} \rangle$ 
to $\langle \mathcal{Z}_{p_N, q_N}, \lambda \rangle$.  
Here, by Proposition~\ref{prop:_inner_automorphisms} 
and assumption~(2) in the first paragraph, 
we can modify $\iota$ and $\zeta$ by inner auromorphisms so that 
the inequalities $\|\alpha^* \circ \iota_{N,n}(g) - \iota \circ \eta(g)\| < 5\gamma$ and 
$\|\zeta \circ \theta_2(g) - \iota \circ \eta'(g)\| < 5\gamma$ for all $g \in G_2$.  
Finally, by Proposition~\ref{prop:_approximate_diagonalizability}, 
we may assume that $\zeta$ is of the form 
\[
\begin{aligned}
	\zeta\bigl(f\bigr)(s) 
		= \operatorname{Ad}(u(s)) \Bigl( \operatorname{diag}
		\bigl[c(f(0)), \dots, c(f(0)), & \\
		f(t_1(s)), \dots, f(t_k(s)), &c(f(1)), \dots, c(f(1)) \bigr] \Bigr), 
\end{aligned}
\]
where $t_1, \dots, t_k$ is the normalized eigenvalue pattern of $\zeta$.  
Since $V(\zeta) < \Delta_g(\gamma)$ for all $g \in \theta_i[G_i]$, 
and since 
\[
\begin{aligned}
	&\|\zeta \circ \theta_2(g) - \alpha^* \circ \iota_{N,n}(g)\| \\
	{}<{}& \|\zeta \circ \theta_2(g) - \iota \circ \eta'(g)\| 
	+ \|\iota \circ \eta'(g) - \iota \circ \eta(g)\| 
	+ \|\iota \circ \eta(g) - \alpha^* \circ \iota_{N,n}(g)\| \\
	{}<{}& 20\gamma, 
\end{aligned}
\]
for all $g \in G_2$, 
one can easily modify the unitary $u$ as in the last paragraph of the proof of 
Proposition~\ref{prop:_inner_automorphisms} so that 
$u$ is constant on $\Sigma_\alpha$ while $\zeta$ still satisfies 
the inequality $\|\zeta \circ \theta_2(g) - \alpha^* \circ \iota_{N,n}(g)\| < 100\gamma$ 
for all $g \in G_2$.  
Then, since $u$ is constant on $\Sigma_\alpha$ and 
$V(\zeta) < \Delta_f(\gamma)$ for all $f \in \theta_1[G_1]$, 
the inequality
\[
	\inf_{\mathclap{g \in \mathcal{Z}_{p_N,q_N}}} 
	\|\zeta \circ \theta_1(f) - \alpha^*(g)\| < \gamma
\]
holds for all $f \in G_1$.  
\[
\begin{original}
\begin{tikzcd}
	 & \langle \mathcal{Z}_{p_n,q_n}, \tau_n' \rangle 
	\arrow[rd,"\theta_2"] \arrow[dd,pos=0.3,"\eta'"] & & \\
	\langle \mathcal{Z}_{p,q}, \tau' \rangle 
	\arrow[rr,pos=0.3,"\theta_1"] & & 
	\langle \mathcal{Z}_{r,s}, \sigma' \rangle 
	\arrow[rd,"\zeta"] \arrow[dd,pos=0.3,"\rho"] & \\
	 & \langle \mathcal{Z}_{p_n,q_n}, \lambda_{p_N,q_N} \rangle 
	\arrow[rr,pos=0.3,"\iota"] & & 
	\langle \mathcal{Z}_{p_N,q_N}, \lambda \rangle \\
	\langle \mathcal{Z}_{p,q}, \tau \rangle 
	\arrow[rr,pos=0.3,"\theta_1"] & & 
	\langle \mathcal{Z}_{r,s}, \sigma \rangle & \\
	 & \langle \mathcal{Z}_{p_n,q_n}, \tau_n \rangle 
	\arrow[uu,pos=0.3,"\eta"] \arrow[ru,"\theta_2"] \arrow[rr,"\iota_{N,n}"]
	 & & \langle \mathcal{Z}_{p_N,q_N}, \tau_N \rangle
	\arrow[uu,"\alpha^*"]
\end{tikzcd}
\end{original}
\begin{arxiv}
\xymatrix{
	 & \langle \mathcal{Z}_{p_n,q_n}, \tau_n' \rangle 
	\ar[rd]^{\theta_2} \ar[dd]_(0.3)\eta & & \\
	\langle \mathcal{Z}_{p,q}, \tau' \rangle 
	\ar[rr]^(0.3){\theta_1} & & 
	\langle \mathcal{Z}_{r,s}, \sigma' \rangle 
	\ar[rd]^\zeta \ar[dd]^(0.3)\rho & \\
	 & \langle \mathcal{Z}_{p_n,q_n}, \lambda_{p_N,q_N} \rangle 
	\ar[rr]^(0.3)\iota & & 
	\langle \mathcal{Z}_{p_N,q_N}, \lambda \rangle \\
	\langle \mathcal{Z}_{p,q}, \tau \rangle 
	\ar[rr]^(0.3){\theta_1} & & 
	\langle \mathcal{Z}_{r,s}, \sigma \rangle & \\
	 & \langle \mathcal{Z}_{p_n,q_n}, \tau_n \rangle 
	\ar[uu]^(0.3)\eta \ar[ru]^{\theta_2} \ar[rr]^{\iota_{N,n}}
	 & & \langle \mathcal{Z}_{p_N,q_N}, \tau_N \rangle
	\ar[uu]^{\alpha^*}
}
\end{arxiv}
\]

Set $\psi := \varphi_{\zeta,\alpha^*}\varphi_{\theta_1,\rho}$.  
Then, for $f \in G_1$ and $g \in G_2$, we have 
\[
\begin{aligned}
	\psi\bigl(f, \iota_{N,n}(g)\bigr) 
	&\leq \varphi_{\theta_1,\rho}\bigl(f, \theta_1(f)\bigr)
	+ \varphi_{\zeta,\alpha^*}\bigl(\theta_1(f),\iota_{N,n}(g)\bigr) \\
	&= \|\theta_1(f) - \rho \circ \theta_1(f)\| 
	+ \|\zeta \circ \theta_1(f) - \alpha^* \circ \iota_{N,n}(g)\| \\
	&\leq \|\zeta \circ \theta_1(f) - \zeta \circ \theta_2(g)\|
	+ 105\gamma \\
	&= \varphi_{\theta_1,\theta_2}(f,g) + 105\gamma.  
\end{aligned}
\]
Also, since $\|\theta_1(f) - \rho \circ \theta_1(f)\| < 5\gamma$ and 
$\inf_g \|\zeta \circ \theta_1(f) - \alpha^*(g)\| < \gamma$, one can easily see that 
$\psi$ is $6\gamma$-total on $G_1$.  
Since $\gamma$ was arbitrary, we may assume $\gamma < \min\{\varepsilon/6, \delta/105\}$ 
so that $\psi$ has the desired property.  
\end{proof}

\begin{rem}\label{rem:_idea_of_the_proof}
Here, for the reader's better understanding, 
we shall present a simpler version of the proof above in a certain special case.  
Let $\langle \mathcal{Z}_{p,q}, \tau \rangle$ be an object of $\mathscr{K}_\mathcal{Z}$, 
$F$ be a finite subset of $\mathcal{Z}_{p,q}$, and $\varphi$ be 
a strict approximate $\mathscr{K}_\mathcal{Z}$-isomorphism from 
$\langle \mathcal{Z}_{p,q}, \tau \rangle$ 
to $\langle \mathcal{Z}_{p_n,q_n}, \tau_n \rangle$.  
Then there exist finite subsets $G_1 \subseteq \mathcal{Z}_{p,q}$ 
and $G_2 \subseteq \mathcal{Z}_{p_n,q_n}$, a joint $\mathscr{K}$-embedding 
$(\theta_1, \theta_2)$ of $\langle \mathcal{Z}_{p,q}, \tau \rangle$ 
and $\langle \mathcal{Z}_{p_n,q_n}, \tau_n \rangle$ into 
some $\langle \mathcal{Z}_{r,s}, \sigma \rangle$ and $\delta > 0$ such that  
\[
	\varphi \geq (\varphi_{\theta_1, \theta_2}|_{G_1 \times G_2})
	|^{\mathcal{Z}_{p,q} \times \mathcal{Z}_{p_n,q_n}} + \delta.  
\]
Without loss of generality, we may assume that $G_1$ includes $F$.  

Now, assume that there happens to be a trace-preserving $*$-homomorphism $\zeta'$ 
from $\langle \mathcal{Z}_{r,s}, \sigma \rangle$ 
to $\langle \mathcal{Z}_{p_N,q_N}, \tau_N \rangle$ for sufficiently large $N$.  
Since $V(\iota_{N,m}) \to 0$ as $N \to \infty$ for each $m$, 
we may assume that both $V(\iota_{N,n})$ and $V(\zeta' \circ \theta_2)$ are 
smaller than $\delta/5$, whence there is a unitary $u$ in $\mathcal{Z}_{p_N,q_N}$ with 
\[
	\bigl\|(\operatorname{Ad}(u) \circ \zeta' \circ \theta_2\bigr)(g) - \iota_{N,n}(g)\| 
	< \delta 
\]
for all $g \in G_2$.  
Now, set $\psi := \varphi_{\zeta' \circ \theta_1}$.  Then we have 
\[
\begin{aligned}
	\psi\bigl(f, \iota_{N,n}(g)\bigr) 
	&= \|\zeta' \circ \theta_1(f) - \iota_{N,n}(g)\| \\
	&\leq \|\zeta' \circ \theta_1(f) - \zeta' \circ \theta_2(g)\| 
	+ \|\zeta' \circ \theta_2(g) - \iota_{N,n}(g)\| \\
	&< \varphi_{\theta_1, \theta_2}(f, g) + \delta, 
\end{aligned}
\]
so $\psi \leq \varphi|^{\mathcal{Z}_{p,q} \times \mathcal{Z}_{p_N,q_N}}$.  
Of course, $\psi$ is $\varepsilon$-total for any $\varepsilon > 0$, 
since clearly 
\[
	\inf_{g \in \mathcal{Z}_{p_N,q_N}} \psi(f, g) = 0
\]
for all $f$.  
This was what we would like to show, in view of Theorem~\ref{thm:_recognition_of_the_limit}.  

In general, $\langle \mathcal{Z}_{r,s}, \sigma \rangle$ is not necessarily embeddable 
into some $\langle \mathcal{Z}_{p_N,q_N}, \tau_N \rangle$, however.  
This is why we need to approximate the measures $\tau_N$ and $\sigma$ by 
$\lambda$ and $\sigma'$ in the original proof above, 
which causes all the other additional steps.   
\end{rem}

In the sequel, we fix an inductive system 
$\{\iota_{n,m} \colon \mathcal{Z}_m \to \mathcal{Z}_n\}$
of prime dimension drop algebras and write its limit by $\mathcal{Z}_0$.  
Note that every $*$-homomorphism between prime dimension drop algebras is 
automatically unital and injective, and $\mathcal{Z}_0$ admits a tracial state.  
We also let $t_1^{m,n} \leq \dots \leq t_{k(m,n)}^{m,n}$ be 
the normalized eigenvalue pattern of $\iota_{n,m}$.  

\begin{lem}\label{lem:_characterization_of_simplicity}
The following two conditions are equivalent.  
\begin{enumerate}
\item
	The limit $\mathcal{Z}_0$ is simple.  
\item
	For any $\varepsilon > 0$, any $y \in [0,1]$ and any $m \in \mathbb{N}$, 
	there exists $n > m$ such that if $x \in [0,1]$ satisfies 
	$t_i^{m,n}(x) = y$ for some $i$, then the Hausdorff distance between 
	$\{t_1^{m,n}(x), \dots, t_{k(m,n)}^{m,n}(x)\}$ and $[0,1]$ is less than $\varepsilon$.  
\end{enumerate}
\end{lem}

\begin{proof}
(1) $\Rightarrow$ (2).  
Suppose that (2) does not hold.  
Then there exist $\varepsilon >0$, $y_0 \in [0,1]$ and $m_0 \in \mathbb{N}$ such that 
for any $n > m_0$ there is $x_n \in [0,1]$ with 
$t_i^{m_0,n}(x_n) = y_0$ for some $i$ and 
\[
	d\bigl(\{t_1^{m_0,n}(x_n), \dots, t_{k(m_0,n)}^{m_0,n}(x_n)\}, [0,1]\bigr) 
	\geq \varepsilon.  
\]
Take $N \in \mathbb{N}$ so that $1/N < \varepsilon/2$.  
For each $n > m_0$ there is $a(n) \in \{0, \dots, N\}$ with 
\[
	U(a(n)/N, 1/N) \cap \{t_1^{m_0,n}(x_n), \dots, t_{k(m_0,n)}^{m_0,n}\} = \varnothing, 
\]
where $U(z, \delta)$ denotes the open ball of center $z$ and radius $\delta$.  
Passing to a subsystem if necessary, we may assume 
that $a(n)$ is constant, say $a$.  Put $U := U(a/N, 1/N)$.  

For each $m, n \in \mathbb{N}$ with $m_0 \leq m \leq n$, set 
\[
\begin{aligned}
	C_{m_0,n} &:= \{ x \in [0,1] \mid t_i^{m_0,n}(x) \notin U \text{ for any } i \}, \\
	C_{m,n} &:= \{ t_i^{m,n}(x) \mid x \in C_{m,n} \}, \\
	C_m &:= \bigcap_{n \geq m} C_{m,n}.  
\end{aligned}
\]
Note that $C_{m_0,n}$ is nonempty, since $x_n$ is in $C_{m_0,n}$.  
Also, if $y$ is in $C_{m,n+1}$, then there exists $x$ in $C_{m_0,n+1}$ with 
$t_i^{m,n+1}(x) = y$ for some $i$.  
Now, since $\iota_{n+1,m} = \iota_{n+1,n} \circ \iota_{n,m}$, 
there are some $j, j'$ with $t_i^{m,n+1}(x) = t_j^{m,n}\bigl(t_{j'}^{n,n+1}(x)\bigr)$.  
On the other hand, $t_i^{m_0,n}\bigl(t_{j'}^{n,n+1}(x)\bigr)$ is not in $U$ for any $i$, 
because $x$ is in $C_{m_0,n+1}$.  
Therefore, $t_{j'}^{n,n+1}(x)$ is in $C_{m_0,n}$, 
whence $y = t_j^{m,n}\bigl(t_{j'}^{n,n+1}(x)\bigr)$ is in $C_{m,n}$.  
Consequently, $C_m$ is a nonempty closed subset of $[0,1]$.  

We shall show
\[
	C_m = \bigcup_{i=1}^{k(m,n)} t_i^{m,n}[C_n].  \tag{$*$}
\]
Clearly, the right-hand side is included in the left-hand side.  
To see the opposite inclusion, let $y$ be in $C_m$.  
Then, for each $l \geq n$, there is $z_l$ in $C_{n,l}$ with 
$t_i^{m,n}(z_l) = y$ for some $i$.  
By the pigeonhole principle, there is $i_0$ with 
$t_{i_0}^{m,n}(z_l) = y$ for infinitely many $l$.  
Let $z$ be a limit point of such $z_l$'s.  
Then clearly $z$ is in $C_n$ and $t_{i_0}^{m,n}(z) = y$.  

For each $m \geq m_0$, set 
\[
	\mathcal{I}_m := \{ f \in \mathcal{Z}_{p_m,q_m} \mid f|_{C_m} \equiv 0 \} 
	\subsetneq \mathcal{Z}_{p_m,q_m}.  
\]
Then, by ($*$), we have 
$\iota_{m+1,m}[\mathcal{Z}_{p_m,q_m}] \cap \mathcal{I}_{m+1} = \iota_{m+1,m}[\mathcal{I}_m]$, 
so the sequence $\{\mathcal{I}_m\}$ defines a closed ideal 
$\mathcal{I}$ of $\mathcal{Z}_0$.  
Since $\mathcal{I}_{m_0}$ includes $\{ f \mid \operatorname{supp} f \subseteq U \}$, 
the ideal $\mathcal{I}$ is nontrivial, so $\mathcal{Z}_0$ is not simple.  

(2) $\Rightarrow$ (1).  
Let $\mathcal{I}$ be a proper ideal of $\mathcal{Z}_0$, and set 
\[
\begin{aligned}
	\mathcal{I}_m &:= \mathcal{I} \cap \mathcal{Z}_{p_m,q_m}, \\
	C_m &:= \{x \mid f(x) = 0 \text{ for all } f \in \mathcal{I}_m \}
\end{aligned}
\]
It suffices to show that $C_m$ coincides with $[0,1]$.  
For this, we may assume $\mathcal{I}_m \subsetneq \mathcal{Z}_{p_m,q_m}$, 
so $C_m$ is nonempty.  Let $y$ be in $C_m$.  
By assumption, for any $\varepsilon > 0$ there is $n_0 > m$ such that 
if $t_i^{m,n_0}(x) = y$, then 
\[
	d\bigl(\{t_1^{m,n_0}(x), \dots, t_{k(m,n_0)}^{m,n_0}(x)\}, [0,1]\bigr) < \varepsilon.  
\]
However, since $C_m = \bigcup_{i=1}^{k(m,n_0)} t_i^{m,n_0}[C_{n_0}]$ by construction, 
we can find $x \in C_{n_0}$ with $t_i^{m,n_0}(x) = y$ for some $i$, and 
\[
	\{ t_1^{m,n_0}(x), \dots, t_{k(m,n_0)}^{m,n_0}(x) \} \subseteq C_m.  
\]
Consequently, it follows that the Hausdorff distance 
between $C_m$ and $[0,1]$ is less than arbitrary $\varepsilon$, so $C_m = [0,1]$.  
\end{proof}

For $y \in [0,1]$ and $\varepsilon > 0$, we set 
\[
\begin{aligned}
	a_{m,n}(y, \varepsilon) :=& \max\{ i \mid \max t_i^{m,n} \leq y+\varepsilon \}, \\
	b_{m,n}(y, \varepsilon) :=& \max\{ i \mid \min t_i^{m,n} < y-\varepsilon \} \\
	c_{m,n}(y,\varepsilon) :=&  \max\{b_{m,n}(y,\varepsilon) - a_{m,n}(y,\varepsilon), 0\} \\
	=& \# \{ i \mid \min t_i^{m,n} < y-\varepsilon \ \& \ 
	\max t_i^{m,n} > y+\varepsilon \}.  
\end{aligned}
\]

\begin{lem}\label{lem:_characterization_of_monotraciality}
The following are equivalent.  
\begin{enumerate}
\item
	The limit $\mathcal{Z}_0$ is monotracial.  
\item
	For any $y$, any $\varepsilon$ and any $m$, 
	\[
		\lim_{n \to \infty} c_{m,n}(y,\varepsilon)/k(m,n) = 0.  
	\]
\end{enumerate}
\end{lem}

\begin{proof}
(1) $\Rightarrow$ (2).  
Suppose (2) does not hold.  
Then, passing to a subsystem if necessary, we may assume that 
there exist $y \in [0,1]$, $\varepsilon > 0$ and $\delta > 0$ with 
$c_{m,n}(y,\varepsilon)/k(m,n) \geq \delta$ for all $n > m$.  
Let $x_{m,n}^1, x_{m,n}^2 \in [0,1]$ be such that 
\[
\begin{aligned}
	t_{a_{m,n}(y,\varepsilon)+1}^{m,n}(x_{m,n}^1) &> y+\varepsilon, & 
	t_{b_{m,n}(y,\varepsilon)+1}^{m,n}(x_{m,n}^2) &< y-\varepsilon, 	
\end{aligned}
\]
and $\tau_1, \tau_2$ be limit points of the tracial states 
$\iota_{n,m}^*(\delta_{x_{m,n}^1}), \iota_{n,m}^*(\delta_{x_{m,n}^2})$ respectively.  
We note that these are restrictions of some tracial states on $\mathcal{Z}_0$.  
Now, if $f \in C[0,1]$ is taken so that 
\[
\begin{aligned}
	f|_{[0,y-\varepsilon]} &\equiv 0,  & f|_{[y+\varepsilon,1]} &\equiv 1, &
	0 &\leq f \leq 1, 
\end{aligned}
\]
then 
\[
\begin{aligned}
	\tau_1 &\geq \varlimsup \bigl[1-a_{m,n}(y,\varepsilon)/k(m,n)\bigr], \\
	\tau_2 &\leq \varliminf \bigl[1 - b_{m,n}(y,\varepsilon)/k(m,n)\bigr], 
\end{aligned}
\]
whence 
\[
	\tau_1(f) - \tau_2(f) 
	\geq \varlimsup c_{m,n}(y,\varepsilon)/k_{m,n}(y,\varepsilon) \geq \delta.  
\]
Consequently, $\mathcal{Z}_0$ is multitracial.  

(2) $\Rightarrow$ (1).  
Suppose (2) holds.  
We shall first show that, given $m \in \mathbb{N}$, $\delta > 0$ and $\varepsilon > 0$, 
one can find $n > m$ with 
\[
	\#\{ i \mid \operatorname{diam} \operatorname{Im} t_i^{m,n} > \delta \}/k(m,n) 
	< \varepsilon.  
\]
Indeed, take $N \in \mathbb{N}$ with $1/N < \delta/3$, 
and let $n(j)$ be sufficiently large so that 
\[
\begin{aligned}
	\frac{c_{m,n(j)}(j/N, 1/N)}{k(m,n(j))} &< \frac{\varepsilon}{N} & 
	(j = 1, \dots, N-1).  
\end{aligned}
\]
Set $n := \max_j n(j)$.  
If $\operatorname{diam}\operatorname{Im}t_i^{m,n} > \delta$, then 
\[
\begin{aligned}
	\min t_i^{m,n} &< (j-1)/N, & \max t_i^{m,n} > (j+1)/N
\end{aligned}
\]
for some $j$, so the desired inequality follows.  

We shall next show that, for $f \in C[0,1]$, $m \in \mathbb{N}$ and $\varepsilon > 0$, 
there exists $n > m$ with 
\[
	\sup_{x,x' \in [0,1]} 
	\bigl|\bigl[\iota_{n,m}^*(\delta_x) - \iota_{n,m}^*(\delta_{x'})\bigr](f)\bigr| 
	\leq \varepsilon.  
\]
For this, we may assume $\|f\| \leq 1$.  Take $\delta > 0$ so that 
$|y-y'| < \delta$ implies $\|f(y)-f(y')\| \leq \varepsilon/3$, 
and put $J := \{ i \mid \operatorname{diam} \operatorname{Im} t_i^{m,n} > \delta \}$.  
By what we proved in the preceding paragraph, 
there exists $n > m$ with $\#J/k(m,n) < \varepsilon/3$.  
Then, for $x,x' \in [0,1]$, we have 
\[
\begin{aligned}
	& \bigl|\bigl[\iota_{n,m}^*(\delta_x) - \iota_{n,m}^*(\delta_{x'})\bigr](f)\bigr| \\
	{}={}& \frac{1}{k(m,n)}
	\Bigl|\sum_i f\bigl(t_i^{m,n}(x)\bigr)-f\bigl(t_i^{m,n}(x')\bigr)\Bigr| \\
	{}\leq{}& \frac{1}{k(m,n)}\Bigl(\sum_{i \in J} + \sum_{i \notin J}\Bigr)
	\bigl|f\bigl(t_i^{m,n}(x)\bigr)-f\bigl(t_i^{m,n}(x)\bigr)\bigr| \\ 
	{}\leq{} &\varepsilon, 
\end{aligned}
\]
as desired.  

Finally, we shall show that $\mathcal{Z}_0$ is monotracial.  
Let $\tau, \tau'$ be tracial states on $\mathcal{Z}_0$ and 
$\mu_\tau, \mu_{\tau'}$ be the corresponding measures on $[0,1]$. 
Fix an element $f$ in the center of $\mathcal{Z}_{p_m,q_m}$, 
which is canonically identified with an element of $C[0,1]$, 
and take sufficiently large $m > n$ so that 
\[
	\sup_{x,x' \in [0,1]} 
	\bigl|\bigl[\iota_{n,m}^*(\delta_x) - \iota_{n,m}^*(\delta_{x'})\bigr](f)\bigr| 
	\leq \varepsilon/3.  
\]
Since the convex combinations of the Dirac measures are weakly* dense in 
the set of probability measures, we can find 
$x_1, \dots, x_l, x_1', \dots, x_l'$ in $[0,1]$ with
\[
\begin{aligned}
	\Bigl|\Bigl(\mu_\tau - \sum_j \delta_{x_j}/l\Bigr)(f \circ t_i^{m,n})\Bigr| 
	&< \varepsilon/3, \\
	\Bigl|\Bigl(\mu_{\tau'} - \sum_j \delta_{x_j'}/l\Bigr)(f \circ t_i^{m,n})\Bigr| 
	&< \varepsilon/3
\end{aligned}
\]
for all $i$.   Consequently, 
\[
\begin{aligned}
	&|\tau(f) - \tau(f')| \\
	{}={}& \frac{1}{k(m,n)}
	\Bigl|\sum_i \mu_\tau(f \circ t_i^{m,n}) - \mu_{\tau'}(f \circ t_i^{m,n})\Bigr| \\
	{}\leq{}& \frac{2}{3}\varepsilon + \frac{1}{k(m,n) \cdot l}
	\Bigl|\sum_{i,j} \delta_{x_j}(f \circ t_i^{m,n}) - 
	\delta_{x_j'}(f \circ t_i^{m,n})\Bigr| \\
	{}={}& \frac{2}{3}\varepsilon + \frac{1}{l}
	\Bigl|\sum_j\bigl[\iota_{n,m}^*(\delta_{x_j}) 
	- \iota_{n,m}^*(\delta_{x_j'})\bigr](f)\Bigr| \\
	{}\leq{}& \varepsilon.   
\end{aligned}
\]
Since $\varepsilon$ was arbitrary, $\tau(f) = \tau'(f)$, and so $\tau = \tau'$.  
\end{proof}

\begin{prop}\label{prop:_simplicity_and_monotraciality}
The limit C*-algebra $\mathcal{Z}_0$ is simple and monotracial 
if and only if $\lim_n V(\iota_{n,m}) = 0$ for each $m$.  
\end{prop}

\begin{proof}
It is clear from Lemmas \ref{lem:_characterization_of_simplicity} 
and~\ref{lem:_characterization_of_monotraciality} that 
if $\lim_n V(\iota_{n,m}) = 0$ for all $m$, then $\mathcal{Z}_0$ is simple and monotracial.  
For the opposite implication, first note that if $\mathcal{Z}_0$ is simple, then 
\[
	\lim_n \operatorname{diam} \operatorname{Im} t_1^{m,n} = 0
\]
for each $m$.  Indeed, for any $\varepsilon > 0$, there exists sufficiently large $n$ 
such that if $t_i^{m,n}(x) = \varepsilon$ for some $i$ and $x \in [0,1]$, then 
\[
	d\bigl(\{t_1^{m,n}(x), \dots, t_{k(m,n)}^{m,n}(x)\}, [0,1]\bigr) < \varepsilon, 
\]
by Lemma~\ref{lem:_characterization_of_simplicity}.  
This implies that $\varepsilon \notin \operatorname{Im} t_1^{m,n}$, 
and since $0 \in \operatorname{Im} t_1^{m,n}$, 
it follows that $\operatorname t_1^{m,n} \subseteq [0,\varepsilon)$.  

Next, for each $m, n \in \mathbb{N}$, 
let $\Delta_{m,n}$ be a map from $(0,1]$ to $(0,1]$ such that 
$|x-x'| \leq \Delta_{m,n}(\varepsilon)$ implies 
$|t_i^{m,n}(x)-t_i^{m,n}(x')| \leq \varepsilon$ for all $i$.  
Passing to a subsystem if necessary, we may assume that 
$\operatorname{Im} t_1^{n,n+1}$ is included in $[0,\Delta_{n-1,n}(\varepsilon/2)]$.  
For a fixed $m \in \mathbb{N}$, set 
\[
	F_n := \operatorname{Im} t_1^{m,m+1} \circ \dots \circ t_1^{n,n+1}
\]
and take $y_0 \in \bigcap_n F_n$.  
By Lemma~\ref{lem:_characterization_of_simplicity}, 
there is $n > m$ such that if $x \in [0,1]$ satisfies $t_i^{m,n}(x) = y_0$ 
for some $i$, then the distance 
between $\{t_1^{m,n}(x), \dots, t_{k(m,n)}^{m,n}(x)\}$ and $[0,1]$ is less than 
$\varepsilon/2$.  
On the other hand, by definition of $F_n$, we can find 
$x \in \operatorname{Im} t_1^{n,n+1} \subseteq [0,\Delta_{n,n+1}(\varepsilon/2)]$ with 
$t_i^{m,n}(x) = y_0$ for some $i$.  
Consequently, it follows that for any $y \in [0,1]$ there exists $i$ with 
$\operatorname{Im} t_i^{m,n} \circ t_1^{n,n+1} \subseteq [y-\varepsilon,y+\varepsilon]$.  

Now, let $f \colon [0,1] \to [0,1]$ be a continuous map such that 
the image of $f$ includes $[y-\varepsilon, y+\varepsilon]$ for some $y \in [0,1]$.  
We shall show the existence of $\delta > 0$ such that 
if a continuous map $g \colon [0,1] \to [0,1]$ satisfies 
$\operatorname{Im} f \circ g \supseteq [y-\varepsilon, y+\varepsilon]$, 
then $\operatorname{diam} \operatorname{Im} g \geq \delta$.  
Indeed, let $(a_n)_n$ and $(b_n)_n$ be 
enumerations of the boundaries of $f^{-1}(y-\varepsilon)$ and $f^{-1}(y+\varepsilon)$ 
respectively.  
If the image of $f \circ g$ includes $[y-\varepsilon, y+\varepsilon]$, 
then the image of $g$ must contain $a_{n'}$ and $b_{n''}$ such that 
there is no $a_n$ or $b_n$ between $a_{n'}$ and $b_{n''}$.  
However, there can be only finitely many such pairs $(n',n'')$, 
because otherwise $f$ cannot be uniformly continuous.  Thus, 
\[
	\delta := \min \bigl\{|a_{n'} - b_{n''}| \bigm| \not\exists n, \ 
	a_{n'} \lessgtr a_n \lessgtr b_{n''} \text{ or } 
	a_{n'} \lessgtr b_n \lessgtr b_{n''} \bigr\}
\]
has the desired property.  

Finally, suppose that there is $m \in \mathbb{N}$ with $\lim_n V(\iota_{n,m}) > 0$.  
Without loss of generality, we may assume $m = 1$.  
Also, by passing to a subsystem if necessary, we may assume that 
there is $y \in [0,1]$ and $\varepsilon > 0$ with the following property: 
For any $n$, there exists $i$ such that the image of $t_i^{1,n}$ includes 
$[y-\varepsilon, y+\varepsilon]$.  
By what we proved in the second paragraph, 
it is easy to find $n_0, i_1, i_2 \in \mathbb{N}$ with
\[
	\operatorname{Im} t_{i_1}^{1,2} \circ t_{i_2}^{2,n_0} \circ t_1^{n_0,n_0+1} 
	\subseteq [y-\varepsilon/2, y+\varepsilon/2].  
\]
We may assume $n_0 = 3$.  Set 
\[
	F := \{ t_h^{1,2} \circ t_i^{2,3} \circ t_j^{3,4} \mid 
	\operatorname{Im} t_h^{1,2} \circ t_i^{2,3} \circ t_j^{3,4} \supseteq 
	[y-\varepsilon, y+\varepsilon] \}
\]
and take $\delta > 0$ so that if $f$ is in $F$ and if the image of $f \circ g$ 
includes $[y-\varepsilon, y+\varepsilon]$, 
then $\operatorname{diam} \operatorname{Im} g \geq \delta$.  
Since $\mathcal{Z}_0$ is monotracial, 
we may assume 
\[
	\#\{t_k^{4,5} \mid \operatorname{diam} \operatorname{Im} t_k^{4,5} \geq \delta \}/k(4,5) 
	< 1/\#F, 
\]
by Lemma~\ref{lem:_characterization_of_monotraciality}.  Then 
\[
\begin{aligned}
	& \#\{ t_h^{1,2} \circ t_i^{2,3} \circ t_j^{3,4} \circ t_k^{4,5} \mid 
	\operatorname{Im}t_h^{1,2} \circ t_i^{2,3} \circ t_j^{3,4} \circ t_k^{4,5} 
	\supseteq [y-\varepsilon, y+\varepsilon] \} \\
	{}\leq{} & \#F \times \#\{t_k^{4,5} \mid \operatorname{diam} t_k^{4,5} \} \leq k(4,5) \\
	{}\leq{} & \#\{ t_h^{1,2} \circ t_i^{2,3} \circ t_j^{3,4} \circ t_k^{4,5} \mid 
	\operatorname{Im}t_h^{1,2} \circ t_i^{2,3} \circ t_j^{3,4} \circ t_k^{4,5} 
	\subseteq [y-\varepsilon/2, y+\varepsilon/2] \}.  
\end{aligned}
\]
However, this implies that there is no $i$ with 
$\operatorname{Im} t_i^{1,5} \supseteq [y-\varepsilon,y+\varepsilon]$, 
which is a contradiction.  
\end{proof}

Combining Propositions \ref{prop:_variation_and_limit} 
and~\ref{prop:_simplicity_and_monotraciality}, we obtain the following result.  

\begin{thm}\label{thm:_characterization_of_jiang_su_algebra}
For an inductive system $\{\iota_{n,m} \colon \mathcal{Z}_m \to \mathcal{Z}_n\}$ of 
prime dimension drop algebras, the following are all equivalent.  
\begin{enumerate}
\item
	The inductive limit of $\{\iota_{n,m} \colon \mathcal{Z}_m \to \mathcal{Z}_n\}$ 
	is isomorphic to $\mathcal{Z}$.  
\item
	The equality $\lim_n V(\iota_{n,m}) = 0$ holds for all $m$.  
\item
	The inductive limit of $\{\iota_{n,m} \colon \mathcal{Z}_m \to \mathcal{Z}_n\}$ 
	is simple and monotracial.  
\end{enumerate}
\end{thm}

It was shown by X.~Jiang and H.~Su that 
every unital $*$-endomorphism of $\mathcal{Z}$ is approximately inner.  
We shall conclude this section by partially recovering this result.  

\begin{prop}\label{prop:_endomorphisms}
Every $\mathscr{K}_\mathcal{Z}$-admissible endomorphism of 
$\langle \mathcal{Z}, \operatorname{tr} \rangle$ is approximately inner.  
\end{prop}

\begin{proof}
Let 
\[
\begin{original}
\begin{tikzcd}
	\langle \mathcal{Z}_{p_1,q_1}, \tau_1 \rangle \arrow[r,"\iota_{2,1}"] & 
	\langle \mathcal{Z}_{p_2,q_2}, \tau_2 \rangle \arrow[r,"\iota_{3,2}"] & 
	\langle \mathcal{Z}_{p_3,q_3}, \tau_3 \rangle \arrow[r,"\iota_{4,3}"] & \dots
\end{tikzcd}
\end{original}
\begin{arxiv}
\xymatrix{
	\langle \mathcal{Z}_{p_1,q_1}, \tau_1 \rangle \ar[r]^{\iota_{2,1}} & 
	\langle \mathcal{Z}_{p_2,q_2}, \tau_2 \rangle \ar[r]^{\iota_{3,2}} & 
	\langle \mathcal{Z}_{p_3,q_3}, \tau_3 \rangle \ar[r]^{\iota_{4,3}} & \dots
}
\end{arxiv}
\]
be a regular sequence with the following property, 
the existence of which follows from Theorem~\ref{thm:_characterization_of_jiang_su_algebra}: 
\begin{enumerate}
\item
	$p_nq_n$ divides $p_{n+1}q_{n+1}$ and $\tau_n$ is atomless for all $n$.  
\item
	For any natural number $a$, there exists sufficiently large $n$ such that 
	$a$ divides $p_nq_n$.  
\end{enumerate}
We shall first show that if $\rho$ is a $\mathscr{K}_\mathcal{Z}$-admissible 
endomorphism of $\langle \mathcal{Z}, \operatorname{tr} \rangle$, 
then for any finite subset $F \subseteq \mathcal{Z}_{p_n,q_n}$ and any $\varepsilon > 0$, 
there exists a morphism $\iota$ 
from $\langle \mathcal{Z}_{p_n,q_n}, \tau_n \rangle$ 
to $\langle \mathcal{Z}_{p_N,q_N}, \tau_N \rangle$ with 
$\|\rho(f) - \iota(f)\| < \varepsilon$ for all $f \in F$.  
Take sufficiently large $m$ and 
so that for any $f \in F$, there exists $f' \in \mathcal{Z}_{p_m,q_m}$ with 
$\|\rho(f) - f'\| < \varepsilon/4$.  
We shall fix such $f'$ for each $f \in F$ and set $F' := \{f' \mid f \in F\}$.  
Put
\[
	\psi := (\varphi_\rho|_{F \times F'})
	|^{\mathcal{Z}_{p_n,q_n} \times \mathcal{Z}_{p_m,q_m}} 
	+ \varepsilon/4
\]
and note that this is a strict approximate $\mathscr{K}_\mathcal{Z}$-isomorphism, 
as $\rho$ is $\mathscr{K}$-admissible.  
Since $\psi$ is strict, there exists a joint $\mathscr{K}$-embedding 
$(\theta_1, \theta_2)$ of $\langle \mathcal{Z}_{p_n,q_n}, \tau_n \rangle$ and 
$\langle \mathcal{Z}_{p_m,q_m}, \tau_m \rangle$ into 
some object $\langle \mathcal{Z}_{r,s}, \sigma \rangle$ with 
$\varphi_{\theta_1,\theta_2} \leq \psi$, 
whence $\|\theta_1(f) - \theta_2(f')\| \leq \varepsilon/2$.  
Now by Proposition~\ref{prop:_embeddability_of_dimension_drop_algebras}, 
one can embed $\mathcal{Z}_{r,s}$ into $\mathcal{Z}_{p_{m'},q_{m'}}$ 
for some $m' > m$.  
By assumption~(2), we may assume that $rs$ divides $p_{m'}q_{m'}$, 
so the remainder indices vanish.  
Consequently, since $\tau_{m'}$ is atomless by assumption~(1), 
one can easily find 
a morphism $\eta$ from $\langle \mathcal{Z}_{r,s}, \sigma \rangle$ 
to $\langle \mathcal{Z}_{p_{m'},q_{m'}}, \tau_{m'} \rangle$.  
Since $V(\iota_{N,m''}) \to 0$ as $N \to \infty$ by 
Theorem~\ref{thm:_characterization_of_jiang_su_algebra}, 
one can find $N > m''$ and a unitary $u$ in $\mathcal{Z}_{p_N,q_N}$ with 
\[
	\bigl\|\bigl(\operatorname{Ad}(u) \circ \iota_{N,m'} \circ 
	\zeta \circ \theta_2\bigr)(f') - \iota_{N,m}(f')\bigr\| < \varepsilon/4
\]
for all $f' \in F'$, by Proposition~\ref{prop:_inner_automorphisms}.  
We set $\iota := \operatorname{Ad}(u) \circ \iota_{N,m'} \circ \zeta \circ \theta_1$.  
Then, for $f \in F$, we have 
\[
\begin{aligned}
	\|\rho(f) - \iota(f)\| 
	&\leq \|\rho(f) - f'\| \\ 
	&+ \bigl\|\iota_{N,m}(f') - 
	\bigl(\operatorname{Ad}(u) \circ \iota_{N,m'} \circ 
	\zeta \circ \theta_2\bigr)(f')\bigr\| \\
	&+ \bigl\|\bigl(\operatorname{Ad}(u) \circ \iota_{N,m'} \circ 
	\zeta \circ \theta_2\bigr)(f') - \bigl(\operatorname{Ad}(u) \circ \iota_{N,m'} \circ 
	\zeta \circ \theta_1\bigr)(f)\bigr\| \\
	&< \varepsilon, 
\end{aligned}
\]
as desired.  

Now, since $V(\iota_{M,N}) \to 0$ as $M \to \infty$, 
there exists sufficiently large $M$ and a unitary $v$ in $\mathcal{Z}_{p_M,q_M}$ with
$\bigl\|\bigl(\operatorname{Ad}(v) \circ \iota\bigr)(f) - f\bigr\| < \varepsilon$ 
for all $f \in F$, by Proposition~\ref{prop:_inner_automorphisms}
This implies $\|\rho(f) - \operatorname{Ad}(v^*)(f)\| < 2\varepsilon$, 
so $\rho$ is approximately inner, which completes the proof.  
\end{proof}

\begin{rem}
It was shown by the author that 
every UHF algebra can be recognized as a Fraïssé category of 
C*-algebras of all matrix-valued functions on cubes with distinguished faithful traces 
and diagonalizable morphisms, 
and that \emph{every} endomorphism of UHF algebra is automatically admissible 
\cite[Theorems 5.4 and~5.10]{masumoto16:_generalized_fraisse}.  
In view of this fact, one should be able to show that 
every endomorphism of $\mathcal{Z}$ is $\mathscr{K}$-admissible, 
although the author could not do that.   
\end{rem}

\noindent
\textbf{Acknowledgement.} The author would like to thank Yuhei Suzuki 
for stimulating conversations.  This work was supported by Research Fellow of the JSPS 
(no.~26--2990) and the Program for Leading Graduate Schools, MEXT, Japan.


\end{document}